\documentclass[11pt]{article}

\usepackage{subfigure}
\usepackage{color}
\usepackage[cmex10]{amsmath}
\usepackage{mathrsfs,amssymb,amsmath}
\usepackage{cases}
\usepackage{graphicx}
\usepackage{caption}

\usepackage{epstopdf}
\usepackage{hyperref}
\hypersetup{pdfborder={0 0 0}}

\usepackage{float}
\usepackage{placeins}

\newcommand{\R}{{\mathbb R}}

\usepackage{algorithm} 
\usepackage{algorithmic} 

\newtheorem{remark}{Remark}[section]
\newtheorem{lemma}{Lemma}[section]

\newtheorem{proposition}{Proposition}

\newtheorem{theorem}{Theorem}[section]
\newtheorem{definition}{Definition}[section]
\newtheorem{notation}{Notation}[section]

\newenvironment{proof}{{\bf {Proof.}}}{\hfill $\square$}
\allowdisplaybreaks[4]

\numberwithin{equation}{section}

\def\pt{\partial}

\def\ra{\rightarrow}

\def\s{\subseteq}

\def\e{\varepsilon}

\def\ol{\overline}

\def\vp{\varphi}
\def\lg{\langle}

\def\rg{\rangle}

\def\bf{\textbf}
\def\pt{\partial}

\def\Om{\Omega}
\def\la{\lambda}
\def\al{\alpha}
\def\be{\beta}
\def\de{\delta}
\def\ga{\gamma}

\def\Ga{\Gamma}

\def\ts{\times}

\def\iy{\infty}

\def\f{\frac}

\def\Lra{\Leftrightarrow}

\def\Span{{\rm{span}}}

\def\df{\mathrm d}

\def\wt{\widetilde}
\def\wh{\widehat}

\def\esssup{\operatorname*{ess\ \! sup}}

\def\hra{\hookrightarrow}

\def\mcA{\mathcal{A}}

	\DeclareMathOperator{\Div}{div}

	\DeclareMathOperator{\dist}{dist}

	\DeclareMathOperator{\supp}{{supp}}

	\newcommand{\N}{\mathbb N}

\begin{document}

\title{Application of the Shape Design Method to Hidden Regularity of Degenerate Hyperbolic Equations with Degenerate Boundary}
\author{
	Dong-Hui Yang$^{1}$, Hongli Sun$^{2^*}$\protect\\
	{\small\itshape\parbox{0.9\textwidth}{\centering $^{1}$School of Mathematics and Statistics, Central South University, Changsha 410083, China}}\protect\\
	{\small\itshape\parbox{0.95\textwidth}{\centering $^{2}$School of Mathematical and Physical Sciences, Chongqing University of Science and Technology, Chongqing 401331, China}}
}

\footnotetext[2]{*Corresponding author: honglisun@126.com}

\date{}
\maketitle{}

\begin{abstract}
 This paper investigates the well-posedness and hidden regularity of boundary-degenerate hyperbolic equations in a two-dimensional setting. The shape design method is employed, which approximates the original degenerate problem by a family of uniformly elliptic problems on truncated subdomains and passes to the limit through uniform estimates. Within this framework, the existence and uniqueness of weak solutions are established in suitable weighted Sobolev spaces. A hidden regularity estimate for the conormal derivative on the nondegenerate portion of the boundary is obtained, providing a uniform bound in terms of the natural weighted energy norms of the initial data and source term. This estimate yields the boundary trace information essential for observability and controllability of degenerate hyperbolic systems. 
 
 \noindent\textbf{Keywords:} degenerate hyperbolic equations; hidden regularity; shape design method; weighted Sobolev spaces; boundary trace estimates

\end{abstract}

\section{Introduction}

Hyperbolic partial differential equations constitute the core mathematical models for describing wave propagation phenomena in elastodynamics, electromagnetics, and acoustics. However, in numerous physically relevant problems, such as those arising in heterogeneous media, layered composite materials, or domains with geometric singularities, the material parameters or wave propagation speeds may degenerate to zero on a portion of the boundary, thereby giving rise to boundary-degenerate hyperbolic problems that cannot be treated within the framework of classical theory. The loss of uniform ellipticity near the degenerate boundary renders standard energy estimates invalid, renders trace theorems inapplicable, and precludes the use of multiplier methods that are essential in boundary control problems.

The analysis of degenerate partial differential equations has attracted considerable attention over the past decades; fundamental results concerning local regularity, Harnack inequalities, and approximation properties were established in \cite{Cavalheiro,CS2,CS3,Fabes,FF}; the corresponding nonlinear potential theory was developed in \cite{Heinonen}; comprehensive weighted-space theories are presented in \cite{GC}; and related inverse problems have also been studied \cite{new1}. However, the literature on degenerate hyperbolic equations remains rather limited; the interplay between hyperbolic dynamics and spatial degeneracy gives rise to new phenomena absent in the elliptic and parabolic cases, and therefore demands fundamentally different analytical tools.

A particularly subtle aspect of hyperbolic problems is the phenomenon of hidden regularity, whereby solutions with only limited interior regularity nonetheless admit a normal derivative trace on the boundary satisfying an \(L^2\)-estimate. This property was studied on \cite{MillaMedeiros,Lions1987} in the context of  hyperbolic equations. Its importance in control theory lies in its role as the foundation for boundary observability estimates, from which exact controllability follows \cite{Coron,Lasiecka,Zuazua,Lions}. In the degenerate setting, however, the hidden regularity question becomes significantly more delicate. Boundary degeneracy complicates the interior regularity analysis and, more fundamentally, raises the question of whether a meaningful trace of the conormal derivative can be defined on the nondegenerate portion of the boundary and whether it admits a uniform estimate in terms of the natural weighted energy norms of the data. These issues are essential prerequisites for boundary control of degenerate hyperbolic systems.

The specific problem addressed in this paper is the following degenerate hyperbolic equation:
\begin{equation}\label{01.10.1}
	\begin{cases}
		\partial_{tt}y - \operatorname{div}(A\nabla y) = f, & \text{in } Q,\\
		y = 0, & \text{on } \partial Q,\\
		y(0) = y^0,\quad \partial_t y(0) = y^1, & \text{in } \Omega,
	\end{cases}
\end{equation}
where $\Omega = (-1,1) \times (0,1)$ and $Q = \Omega \times (0,T)$ for some $T>0$, and the coefficient matrix is given by $A = \operatorname{diag}(1, x_2^\alpha)$ with $\alpha \in (0,1)$. Here the weight function is $w = x_2^\alpha$, and the principal operator takes the form $\mathcal{A}y = -\operatorname{div}(A\nabla y)$. To formulate the boundary trace problem precisely, we decompose the boundary as follows. Let $\Gamma_2^0 = \{x = (x_1,x_2) \in \overline{\Omega} : x_2 = 0\}$ denote the degenerate portion of the boundary, and set $\Gamma = \partial\Omega \setminus \Gamma_2^0$ for the nondegenerate portion. For each $\delta \in (0, \frac{1}{4})$, we introduce the truncated domain 
$\Omega_\delta = \{x \in \Omega : x_2 > \delta\}$ with its associated boundaries 
\begin{equation*}
	\Gamma_2^\delta = \{x \in \overline{\Omega}_\delta : x_2 = \delta\}
	\quad\text{and}\quad
	\Gamma_\delta = \partial\Omega_\delta \setminus \Gamma_2^\delta.
\end{equation*}
Finally, the conormal derivative associated with the degenerate operator is defined by
\begin{equation*}
	\frac{\partial u}{\partial \nu_A} = A\nabla u \cdot \nu \quad \text{on } \partial\Omega,
\end{equation*}
where $\nu$ denotes the outward unit normal.

Various analytical techniques have been developed for degenerate PDEs. Carleman estimates \cite{Cannarsa1,Araruna,FG,Fursikov,Lebeau} and spectral inequalities \cite{Buffe1,Miller} have proven effective for controllability problems \cite{Beauchard,Beauchard1,Lu,Weiss}. However, they typically rely on delicate pointwise estimates in weighted function spaces and are not directly suited for investigating hidden regularity. For boundary trace analysis of hyperbolic equations, multiplier-based identities provide a natural framework \cite{Lasiecka}. Instead of following this classical multiplier approach, the present paper adopts an alternative perspective via the shape design method. This technique originates from the classical theory of PDE-constrained shape optimization \cite{Chenais,Buttazzo,Henrot,Tiba,Wang,He,Guo1,Guo2}. Rather than seeking an optimal domain, the method approximates the original degenerate problem by a family of uniformly elliptic problems defined on truncated subdomains; degeneracy is removed at the level of these approximate problems, for which classical estimates become available. This strategy has been successfully applied to establish well-posedness and controllability for degenerate PDEs \cite{Guo,Yangzhong,YangGuoz,Yang,Yang1,Wu}. For hidden regularity, where boundary-trace information must be extracted from solutions with limited interior regularity, this approximation framework is particularly advantageous, since it circumvents the requirement for high-order interior regularity estimates in weighted spaces.

In this paper, we employ the shape design method to address the well-posedness and hidden regularity of the degenerate hyperbolic equation \eqref{01.10.1}. The approximation proceeds by replacing the original problem on \(\Omega\) with a family of uniformly elliptic problems on the truncated domains \(\Omega_\delta\). On each \(\Omega_\delta\), classical well-posedness and trace theory apply. The central difficulty lies in establishing uniform estimates in \(\delta\), thereby enabling passage to the limit \(\delta \to 0^+\) and recovery of the desired boundary estimates for the original solution. This approach circumvents the need for higher-order interior regularity in weighted spaces, rendering it particularly effective for hidden regularity analysis. The main contribution of this paper is to establish a hidden regularity estimate for the boundary-degenerate hyperbolic equation under consideration, providing a uniform bound on the conormal derivative on the nondegenerate boundary \(\Gamma\) in terms of the natural weighted energy norms of the data. This estimate delivers the precise boundary trace information required for multiplier methods in observability and controllability theories for degenerate hyperbolic systems. Beyond this specific result, the methodology developed herein furnishes a general framework applicable to a broader class of degenerate hyperbolic problems, including those with time-dependent coefficients, nonlinearities, or more general degeneracy structures in higher dimensions.

The remainder of the paper is organized as follows. In Section 2, we introduce the weighted Sobolev framework and establish the well-posedness of the original degenerate equation, including the energy estimates and the preliminary trace regularity on a subdomain away from the degeneracy. In Section 3, we develop the shape design approximation for the hyperbolic setting: we prove well-posedness and uniform estimates for the truncated problems, establish the convergence of the approximate solutions and their normal derivatives, and ultimately derive the main hidden regularity theorem for the original problem.

\section{Weighted framework and well-posedness}

\subsection{Solution spaces}\label{S2}

Denote 
\begin{equation*}
	H^1(\Om;w)=\left\{u\in L^2(\Om)\colon \int_\Om A\nabla u\cdot \nabla u\df x<+\iy\right\},
\end{equation*}
its inner product and norm are defined by 
\begin{equation*}
	(u,v)_{H^1(\Om;w)}=\int_\Om \left(uv+\nabla u\cdot A\nabla v\right)\df x, \quad \|u\|_{H^1(\Om;w)}=(u,u)_{H^1(\Om;w)}^\f{1}{2}. 
\end{equation*}
Set
\begin{equation*}
	H_0^1(\Om;w)=\mbox{the closure of $C_0^\iy(\Om)$ in } H^1(\Om;w). 
\end{equation*}

Denote
\begin{equation*}
	H^2(\Om;w)=\left\{u\in H^1(\Om;w)\colon \mcA u\in L^2(\Om)\right\}, 
\end{equation*}
its inner product and norm are defined by
\begin{equation*}
	(u,v)_{H^2(\Om;w)}=(u,v)_{H^1(\Om;w)}+\int_\Om (\mcA u)(\mcA v)\df x,\quad \|u\|_{H^2(\Om;w)}=(u,u)_{H^2(\Om;w)}^\f{1}{2}. 
\end{equation*}

It is well known that the spaces 
\begin{equation*}
	H_0^1(\Om;w),\quad H^1(\Om;w),\quad H^2(\Om;w)
\end{equation*}
are Hilbert spaces (See \cite{GC,Heinonen}). 

Denote 
\begin{equation*}
	D(w)=H^2(\Om;w)\cap H_0^1(\Om;w). 
\end{equation*}

\begin{lemma}\label{03.11.L1}
	Let $\al\in (0,1)$. Then for all $u\in H_0^1(\Om;w)$ we have
	\begin{equation*}
		\int_\Om x_2^{\al-2}u^2\df x\leq \f{4}{(1-\al)^2}\int_\Om x_2^\al   (\pt_{x_2}u)^2\df x. 
	\end{equation*}
\end{lemma}

\begin{proof}
	Let $u\in H_0^1(\Om;w)$. Then for each $\be\in (\al,1)$ we have 
	\begin{equation*}
		\begin{split}
			\int_s^1x_2^{\al-1-\be}\df x_2
			&=\f{1}{\al-\be}\left(1-s^{\al-\be}\right)\leq \f{1}{\be-\al}s^{\al-\be}, 
		\end{split}
	\end{equation*}
	and then 
	\begin{equation*}
		\begin{split}
			\int_\Om x_2^{\al-2}u^2\df x
			&=\int_{-1}^1\int_0^1 x_2^{\al-2}u^2\df x_2\df x_1=\int_{-1}^1\int_0^1x_2^{\al-2}|u(x_1,x_2)-u(x_1,0)|^2\df x_2\df x_1\\
			&=\int_{-1}^1\int_0^1x_2^{\al-2}\left(\int_0^{x_2}\pt_{x_2}u(x_1,s)\df s\right)^2\df x_2\df x_1\\
			&\leq \int_{-1}^1\int_0^1x_2^{\al-2}\left(\int_0^{x_2}s^\be\left|\pt_{x_2}u(x_1,s)\right|^2\df s\right)\left(\int_0^{x_2}s^{-\be}\df s\right)\df x_2\df x_1\\
			&=\f{1}{1-\be}\int_{-1}^1\int_0^1\int_0^{x_2} x_2^{\al-1-\be} s^\be \left|\pt_{x_2}u(x_1,s)\right|^2\df s\df x_2\df x_1\\
			&=\f{1}{1-\be}\int_{-1}^1\int_0^1\int_{s}^1x_2^{\al-1-\be}s^\be \left|\pt_{x_2}u(x_1,s)\right|^2\df x_2\df s\df x_1\\
			&\leq \f{1}{(\be-\al)(1-\be)}\int_{-1}^1\int_0^1 s^\al \left|\pt_{x_2}u(x_1,s)\right|^2\df s\df x_1\\
			&=\f{1}{(\be-\al)(1-\be)}\int_\Om x_2^\al \left|\pt_{x_2}u\right|^2\df x. 
		\end{split}
	\end{equation*}
	Taking $\be=\f{1+\al}{2}$,  the lemma is obtained. 
\end{proof}

\begin{remark}\label{03.11.R1}
	Let $\al\in (0,1)$. Then for all $u\in H_0^1(\Om;w)$ we have 
	\begin{equation}\label{03.11.1}
		\int_\Om u^2\df x=\int_\Om x_2^{2-\al}x_2^{\al-2}u^2\df x\leq \f{4}{(1-\al)^2}\int_\Om x_2^\al  (\pt_{x_2}u)^2\df x. 
	\end{equation}
	Hence the inner product and the norm are defined by 
	\begin{equation*}
		(u,v)_{H_0^1(\Om;w)}=\int_\Om A\nabla u\cdot \nabla u\df x,\quad \|u\|_{H_0^1(\Om;w)}=(u,u)_{H^1(\Om;w)}^\f{1}{2}
	\end{equation*}
	are equivalent inner product and norm on $H_0^1(\Om;w)$, respectively. 
\end{remark}

\begin{lemma}\label{03.11.L2}
	The embedding $H_0^1(\Om;w)\hra L^2(\Om)$ is compact. 
\end{lemma}

\begin{proof}
	Let $\{u_n\}_{n\in\N}\s H_0^1(\Om;w)$ is a bounded sequence. i.e., $\|u_n\|_{H_0^1(\Om;w)}\leq C_0$ for some positive constant $C_0$. Then there exists a subsequence of $\{u_n\}_{n\in\N}$, still denoted by itself, and $u_0\in H_0^1(\Om;w)$ such that $u_n\ra u_0$ weakly in $H_0^1(\Om;w)$ as $n\ra\iy$. Now, we shall prove that there exists a subsequence $\{u_{n_k}\}_{k\in\N}$ of $\{u_n\}_{n\in\N}$ such that $u_{n_k}\ra u_0$ strongly in $L^2(\Om)$ as $k\ra\iy$. We assume $u_0=0$, for otherwise, we replace $u_n$ by  $u_{n}-u_0$ for all $n\in\N$.

	Let $\e>0$. For each $\de\in (0,\f{1}{4})$, denote $\Om_\de=(-1,1)\ts (\de,1)$, from Lemma \ref{03.11.L1}, we have 
	\begin{equation*}
		\int_{\Om-\Om_\de}u^2\df x=\int_{\Om-\Om_\de} x_2^{2-\al}x_2^{\al-2}u^2\df x\leq \de^{2-\al}\f{4}{(1-\al)^2}\int_\Om x_2^\al \left|\pt_{x_2}u\right|^2\df x\leq \de^{2-\al}\f{4C_0^2}{(1-\al)^2}. 
	\end{equation*}
	Choosing $\de_0>0$ small enough, for all $\de\in (0,\de_0]$, we have $\de^{2-\al}\f{4C_0^2}{(1-\al)^2}\leq \f{\e^2}{4}$. 
	
	Note that for fixed $\de_0$, the embedding $H^1(\Om_{\de_0};w_{\de_0})=H^1(\Om_{\de_0})\hra L^2(\Om_{\de_0})$ is compact by the classical Sobolev embedding theorem, hence, there exists a $k=k(\de_0)\in \N$ such that $u_{n_{k}}$ satisfies $\|u_{n_k}\|_{L^2(\Om_{\de_0})}<\f{\e}{2}$, since $u_n\ra 0$ weakly in $L^2(\Om_{\de_0})$.  Hence, $\|u_{n_k}\|_{L^2(\Om)}<\e$. This complete the proof of this lemma. 
\end{proof}

\subsection{Spectrum}

Combining Remark \ref{03.11.R1}, Lemma \ref{03.11.L2}, and Lemma 3.1 of \cite{YangGuo1},  we obtain the discrete spectrum of the partial differential operator $\mcA$:
\begin{equation*}
	0< \la_1<\la_2\leq \la_3\leq \cdots \ra +\iy. 
\end{equation*}
Note that 
\begin{equation*}
	\f{(1-\al)^2}{4}\leq \la_1=\inf_{0\neq u\in H_0^1(\Om;w)}\f{\int_\Om A\nabla u\cdot \nabla u \df x}{\int_\Om u^2\df x}. 
\end{equation*}

\begin{notation}\label{03.11.N1}
	Let $\Phi_n\ (n\in\N)$ be the eigenfunction of $\mcA$ with respect to the eigenvalue $\la_n$. i.e., 
	\begin{equation}\label{03.11.2}
		\begin{cases}
			\mcA\Phi_n=\la_n\Phi_n, &\mbox{in }\Om,\\
			\Phi_n=0, &\mbox{on }\pt\Om. 
		\end{cases}
	\end{equation}
	Moreover, set $\|\Phi_n\|_{L^2(\Om)}=1$ for all $n\in\N$, then  $\{\Phi_n\}_{n\in\N}$ is a orthonormal basis of $L^2(\Om)$, and  $\{\Phi_n\}_{n\in\N}$ is the orthonormal basis of $L^2(\Om)$, moreover, $\{\Phi_n\}_{n\in\N}$ is an orthogonal subset of $H_0^1(\Om;w)$. (See \cite[Theorem 7 (pp. 728) in Appendix D]{Evans}, or the proof of Lemma \ref{06.28.L1} in the following.) 
\end{notation}

\begin{lemma}\label{06.28.L1}
	Let $u=\sum\limits_{i=1}^\iy u_i \Phi_i\in H_0^1(\Om;w)$ with $u_i=(u,\Phi_i)_{L^2(\Om)}$ for all $i\in\N$. We have $\nabla u=\sum\limits_{i=1}^\iy u_i\nabla \Phi_i$ and $\|u\|_{H_0^1(\Om;w)}=(\sum\limits_{i=1}^\iy u_i^2\la_i)^\f{1}{2}$, and
	\begin{equation*}
		u\in H^2(\Om;w)\Lra \sum_{i=1}^\iy u_i^2\la_i^2<\iy,
	\end{equation*}
	and
	\begin{equation*}
		\mcA u=\sum_{i=1}^\iy u_i\la_i\Phi_i \mbox{ and } \|\mcA u\|_{L^2(\Om)}=\left(\sum_{i=1}^\iy u_i^2\la_i^2\right)^\f{1}{2}.
	\end{equation*}
\end{lemma}

\begin{proof}
	From \eqref{03.11.2}, we have
	\begin{equation}\label{06.04.3}
		\int_\Om (\nabla \Phi_k\cdot\nabla \Phi_l)w\df x=\de_{kl}\la_k \mbox{ for } k,l\in \N,
	\end{equation}
	where $\de_{kl}$ is the Kronecker delta function, defined as $\de_{kl}=1$ for $k=l$ and $\de_{kl}=0$ for $k\neq l$.
	We now demonstrate that $\{\la_k^{-\f{1}{2}}\Phi_k\}_{k=1}^\iy$ forms an orthonormal basis of $H_0^1(\Om;w)$.
	From \eqref{06.04.3}, it is straightforward to verify that $\{\la_k^{-\f{1}{2}}\Phi_k\}_{k=1}^\iy$ is an orthonormal subset of $H_0^1(\Om;w)$. To prove it is a basis, assume by contradiction that there exists $0\neq u\in H_0^1(\Om;w)$ such that
	\begin{equation*}
		\int_\Om (\nabla \Phi_k\cdot \nabla u)w\df x=0 \mbox{ for all }k\in\N.
	\end{equation*}
	Given that $\{\Phi_k\}_{k\in\N}$ is an orthonormal basis of $L^2(\Om)$, for $u\in H_0^1(\Om;w)$, we can express
	\begin{equation}\label{06.04.4}
		u=\sum_{k=1}^\iy d_k\Phi_k \mbox{ where }  d_k=(u, \Phi_k)_{L^2(\Om)}, k\in\N.
	\end{equation}
	Then  we have
	\begin{equation*}
		0=(u,\Phi_k)_{H_0^1(\Om;w)}=\int_\Om (\nabla \Phi_k\cdot \nabla u)w\df x=\la_k\int_\Om \Phi_ku\df x =\la_kd_k,
	\end{equation*}
	which implies $d_k=0$. This leads to $u=0$, a contradiction. 
	
	From above, since $u\in H_0^1(\Om;w)$,  we have 
	\begin{equation*}
		\nabla u=\sum_{i=1}^\iy e_i\la_i^{-\f{1}{2}}\nabla\Phi_k, \mbox{ and } \|u\|_{H_0^1(\Om;w)}^2=\sum_{i=1}^\iy e_i^2,
	\end{equation*}
	and 
	\begin{equation*}
		e_i=\int_\Om \left(\nabla u\cdot \la_i^{-\f{1}{2}}\nabla\Phi_i\right) w\df x=\la_i^{-\f{1}{2}}\int_\Om (\nabla u\cdot \nabla \Phi_i)w\df x=\la_i^\f{1}{2}\int_\Om u\Phi_i\df x=\la_i^\f{1}{2}u_i, 
	\end{equation*} 
	hence,  $\nabla u=\sum\limits_{i=1}^\iy u_i\nabla\Phi_i$. Moreover, $\|u\|_{H_0^1(\Om;w)}=(\sum\limits_{i=1}^\iy \la_iu_i^2)^\f{1}{2}$. 
	
	Let $u\in H^2(\Om;w)$. Take $\vp_n=\sum\limits_{i=1}^n u_i\Phi_i\in H_0^1(\Om;w)\cap H^2(\Om;w)$ for each $n\in\N$, then from
	\begin{equation*}
		\begin{split}
			(\mcA u, \mcA\vp_n)_{L^2(\Om)}
			&=\sum_{i=1}^n u_i\la_i \int_\Om (\mcA u)\Phi_i\df x=\sum_{i=1}^n u_i\la_i\int_\Om u\mcA\Phi_i\df x\\
			&=\sum_{i=1}^n u_i\la_i^2\int_\Om u\Phi_i\df x=\sum_{i=1}^n u_i^2\la_i^2
		\end{split}
	\end{equation*}
	and $\|\mcA\vp_n\|_{L^2(\Om)}^2=\sum_{i=1}^n u_i^2\la_i^2$ we obtain
	\begin{equation*}
		\sum_{i=1}^nu_i^2\la_i^2\leq \|\mcA u\|_{L^2(\Om)}^2
	\end{equation*}
	for all $n\in\N$ by Cauchy inequality. This implies that $\sum_{i=1}^\iy u_i^2\la_i^2\leq \|\mcA u\|_{L^2(\Om)}^2<\iy$.

	Let $\sum\limits_{i=1}^\iy u_i^2\la_i^2<\iy$. For each  $\vp\in C_0^\iy(\Om)$,  we have
	\begin{equation*}
		\begin{split}
			(\mcA u, \vp)_{L^2(\Om)}
			&=\int_\Om (\nabla u\cdot \nabla\vp)w\df x=\sum_{i=1}^\iy u_i \int_\Om (\nabla\Phi_i\cdot \nabla\vp)w\df x=\sum_{i=1}^\iy u_i\la_i\int_\Om \Phi_i\vp\df x.
		\end{split}
	\end{equation*}
	Note that
	\begin{equation*}
		\int_\Om \left(\sum_{i=1}^n u_i\la_i\Phi_i\right)^2\df x=\sum_{i=1}^n u_i^2\la_i^2\leq \sum_{i=1}^\iy u_i^2\la_i^2<\iy \mbox{ for all } n\in\N,
	\end{equation*}
	i.e., $\sum\limits_{i=1}^\iy u_i\la_i\Phi_i\in L^2(\Om)$, Hence
	\begin{equation*}
		(\mcA u, \vp)_{L^2(\Om)}=\left(\sum_{i=1}^\iy u_i\la_i\Phi_i, \vp\right)_{L^2(\Om)}.
	\end{equation*}
	This implies that $\|\mcA u\|_{L^2(\Om)}\leq \sum\limits_{i=1}^\iy u_i^2\la_i^2$ and $\mcA u=\sum\limits_{i=1}^\iy u_i\la_i\Phi_i$.
\end{proof}

\subsection{Existence of weak solution}

\begin{definition}\label{03.11.D1}
	Let $y^0\in H_0^1(\Om;w), y^1\in L^2(\Om)$ and $f\in L^2(Q)$. We call 
	\begin{equation*}
		y\in L^2(0,T; H_0^1(\Om;w))\cap H^1(0,T; L^2(\Om))\cap H^2(0,T; H^{-1}(\Om;w))
	\end{equation*}
	is a weak solution of  \eqref{01.10.1} with respect to $(y^0,y^1,f)$, provided 
	
	(i) for each $v\in H_0^1(\Om;w)$ and a.e.~$t\in [0,T]$, we have
	\begin{equation*}
		\lg\pt_{tt}y, v\rg_{H^{-1}(\Om;w), H_0^1(\Om;w)}+(y, v)_{H_0^1(\Om;w)}=(f,v)_{L^2(\Om)}, 
	\end{equation*}
	
	(ii) $y(0)=y^0$ and $\pt_ty(0)=y^1$. 
\end{definition}

\begin{theorem}\label{03.11.T1}
	Let $y^0\in H_0^1(\Om;w)$, $y^1\in L^2(\Om)$ and $f\in L^2(Q)$. Then there exists a unique weak solution 
	\begin{equation*}
		y\in L^2(0,T; H_0^1(\Om;w))\cap H^1(0,T; L^2(\Om))\cap H^2(0,T; H^{-1}(\Om;w))
	\end{equation*}
	for the equation \eqref{01.10.1} with respect to $(y^0,y^1, f)$, and 
	\begin{equation}\label{03.12.1}
		\begin{split}
			&\esssup_{t\in [0,T]}\left(\|y(t)\|_{H_0^1(\Om;w)}+\|\pt_ty\|_{L^2(\Om)}\right)+\|\pt_{tt}y\|_{L^2(0,T; H^{-1}(\Om;w))}+\left\|\f{\pt y}{\pt \nu}\right\|_{L^2(0,T; L^2(\Ga\cap\pt\Om_\f{1}{8}))}\\
			&\leq C\left(\|y^0\|_{H_0^1(\Om;w)}+\|y^1\|_{L^2(\Om)}+\|f\|_{L^2(Q)}\right), 
		\end{split}
	\end{equation}
	where the positive constant $C$ depending only on $\al$ and $T$. 
	
	Furthermore, if in addition $y^0\in D(\mcA)$, and $y^1\in H_0^1(\Om;w)$ and $f\in H^1(0,T; L^2(\Om))$, then 
	\begin{equation}\label{03.12.2}
		\begin{split}
			&\esssup_{t\in [0,T]}\left(\|y\|_{H^2(\Om_\f{1}{8})}+\|y\|_{D(\mcA)}+\|\pt_ty\|_{H_0^1(\Om;w)}+\|\pt_{tt}y\|_{L^2(\Om)}\right)+\left\|\f{\pt (\pt_ty)}{\pt \nu}\right\|_{L^2(0,T; L^2(\Ga\cap\pt\Om_\f{1}{8}))} \\
			&\leq C\left(\|y^0\|_{D(\mcA)}+\|y^1\|_{H_0^1(\Om;w)}+\|f\|_{H^1(0,T; L^2(\Om))}\right), 
		\end{split}
	\end{equation}
	where the positive constant $C$ depending only on $\al$ and $T$. 
\end{theorem}

\begin{proof}
	We prove this theorem by the following steps. 
	
	{\it Step 1}. Galerkin method. 
	
	Let $n\in\N$. Denote 
	\begin{equation}\label{03.11.3}
		y^m(x,t)=\sum_{n=1}^m y_n^m(t)\Phi_n(x), \quad f^m(x,t)=\sum_{n=1}^m f_n^m(t)\Phi_n(x), 
	\end{equation}
	where $y_n^m(t), n=1,\cdots, m$ satisfies the following equation 
	\begin{equation}\label{03.11.4}
		\f{\df^2}{\df t^2}y_n^m(t)+\la_ny_n^m(t)=f_n^m(t), \mbox{ for } t\in [0,T], n=1,\cdots, m, 
	\end{equation}
	with initial data 
	\begin{equation}\label{03.11.5}
		\begin{split}
			y_n^m(0)=(y^0,\Phi_n)_{L^2(\Om)}, \quad \f{\df y_n^m}{\df t}(0)=(y^1, \Phi_n)_{L^2(\Om)}, \quad f_n^m(t)=(f,\Phi_n)_{L^2(\Om)}, \ n=1,\cdots,m. 
		\end{split}
	\end{equation}
	It is obviously that \eqref{03.11.4} is equivalent to the following equation
	\begin{equation}\label{03.11.6}
		(\pt_{tt}y^m, \Phi_n)_{L^2(\Om)}+(y^m,\Phi_n)_{H_0^1(\Om;w)}=(f^m,\Phi_n)_{L^2(\Om)}, \ n=1,\cdots, m. 
	\end{equation}

	{\it Step 2}. Energy estimate.
	
	Multiplying \eqref{03.11.6} by $\f{\df}{\df t}y_n^m$, summing $n=1,\cdots,m$, then 
	\begin{equation*}
		(\pt_{tt}y^m,\pt_ty^m)_{L^2(\Om)}+(y^m, \pt_ty^m)_{H_0^1(\Om;w)}=(f^m,\pt_ty^m)_{L^2(\Om)}. 
	\end{equation*}
	Note that 
	\begin{equation*}
		(\pt_{tt}y^m,\pt_ty^m)_{L^2(\Om)}=\f{1}{2}\f{\df }{\df t} \|\pt_ty^m\|_{L^2(\Om)}^2, \mbox{ and } (y^m,\pt_ty^m)_{H_0^1(\Om;w)}=\f{1}{2}\f{\df}{\df t}\|y^m\|_{H_0^1(\Om;w)}^2, 
	\end{equation*}
	we get
	\begin{equation*}
		\f{\df}{\df t}\left(\|\pt_ty^m\|_{L^2(\Om)}^2+\|y^m\|_{H_0^1(\Om;w)}^2\right)\leq \|f^m\|_{L^2(\Om)}^2+\|\pt_ty^m\|_{L^2(\Om)}^2. 
	\end{equation*}
	From the Gronwall's inequality \cite[B.~2 (j) in Appendix B, p.~707]{Evans}, we get
	\begin{equation*}
		\begin{split}
			\|\pt_ty^m\|_{L^2(\Om)}^2+\|y^m\|_{H_0^1(\Om;w)}^2\leq e^t\left(\|\pt_ty^m(0)\|_{L^2(\Om)}^2+\|y^m(0)\|_{H_0^1(\Om;w)}^2+\|f\|_{L^2(Q)}^2\right)
		\end{split}
	\end{equation*}
	for all $t\in [0,T]$. Together with this and \eqref{03.11.5} and Lemma \ref{06.28.L1}, we get
	\begin{equation*}
		\|\pt_ty^m\|_{L^2(\Om)}^2+\|y^m\|_{H_0^1(\Om;w)}^2\leq e^t\left(\|y^0\|_{H_0^1(\Om;w)}^2+\|y^1\|_{L^2(\Om)}^2+\|f\|_{L^2(Q)}^2\right)
	\end{equation*}
	for all $t\in [0,T]$. i.e., 
	\begin{equation}\label{03.11.7}
		\esssup_{t\in [0,T]}\left(	\|\pt_ty^m\|_{L^2(\Om)}^2+\|y^m\|_{H_0^1(\Om;w)}^2\right)\leq e^T\left(\|y^0\|_{H_0^1(\Om;w)}^2+\|y^1\|_{L^2(\Om)}^2+\|f\|_{L^2(Q)}^2\right). 
	\end{equation}
	
	Now, for each $v\in H_0^1(\Om;w)$ with $\|v\|_{H_0^1(\Om;w)}\leq 1$, we write $v=v^1+v^2$ with $v^1\in \Span\{\Phi_n\}_{n=1}^m$ and $v^2\bot\Phi_n=0$ for all $n=1,\cdots, m$, then, from 
	\begin{equation*}
		\begin{split}
			\lg\pt_{tt}y^m,v\rg_{H^{-1}(\Om;w), H_0^1(\Om;w)}
			&=(\pt_{tt}y^m, v)_{L^2(\Om)}=(\pt_{tt}y^m,v^1)_{L^2(\Om)}=(f^m,v^1)_{L^2(\Om)}-(y^m,v^1)_{H_0^1(\Om;w)}
		\end{split}
	\end{equation*}
	and Remark \ref{03.11.R1}, 
	we get 
	\begin{equation*}
		\begin{split}
			\|\pt_{tt}y^m\|_{H^{-1}(\Om;w)}
			&\leq \f{2}{1-\al}\|f^m\|_{L^2(\Om)}+\|y^m\|_{H_0^1(\Om;w)}
		\end{split}
	\end{equation*}
	by $\|v^1\|_{H_0^1(\Om;w)}\leq 1$. This shows that 
	\begin{equation}\label{03.11.8}
		\begin{split}
			\|\pt_{tt}y^m\|_{L^2(0,T; H^{-1}(\Om;w))}\leq C\left(\|y^0\|_{H_0^1(\Om;w)}+\|y^1\|_{L^2(\Om)}+\|f\|_{L^2(Q)}\right)
		\end{split}
	\end{equation}
	by \eqref{03.11.7}, where the positive constant $C$ depending only on $\al$ and $T$. 
	
	{\it Step 3}. Approximation, or weakly convergence. 
	
	From \eqref{03.11.7} and  \eqref{03.11.8}, there exists 
	\begin{equation*}
		z\in L^2(0,T; H_0^1(\Om;w))\cap H^1(0,T; L^2(\Om))\cap H^2(0,T; H^{-1}(\Om;w))
	\end{equation*}
	such that 
	\begin{equation}\label{03.11.9}
		\begin{split}
			y^m
			&\ra z\hspace{4mm} \mbox{ weak star in  } L^\iy(0,T; H_0^1(\Om;w)), \\
			\pt_ty^m
			&\ra \pt_t z\hspace{1mm} \mbox{ weak star in  } L^\iy(0,T; L^2(\Om)), \\
			\pt_{tt}y^m
			&\ra \pt_{tt}z \mbox{ weakly in } L^2(0,T; H^{-1}(\Om;w)). 
		\end{split}
	\end{equation}
	Moreover, we have
	\begin{equation}\label{03.11.10}
		\begin{split}
			&\esssup_{t\in [0,T]}\left(\|\pt_tz\|_{L^2(\Om)}^2+\|z\|_{H_0^1(\Om;w)}\right)+\|\pt_{tt}z\|_{L^2(0,T; H^{-1}(\Om;w))}\\
			&\leq C\left(\|y^0\|_{H_0^1(\Om;w)}+\|y^1\|_{L^2(\Om)}+\|f\|_{L^2(Q)}\right), 
		\end{split} 
	\end{equation}
	where the positive constant $C$ depending only on $\al$ and $T$. Moreover, 
	\begin{equation}\label{03.11.11}
		z\in C([0,T];L^2(\Om)), \mbox{ and } \pt_t z\in C([0,T]; H^{-1}(\Om;w)). 
	\end{equation}
	
	{\it Step 4}. We prove $z$ is a weak solution of the equation \eqref{01.10.1}. 
	
	Let $k\in\N$, and 
	\begin{equation}\label{03.11.13}
		\psi=\sum_{n=1}^k\psi_n(t)\Phi_n(x) \mbox{ with } \psi_n(t)\in C^2[0,T], n=1,\cdots, k. 
	\end{equation}
	Multiplying \eqref{03.11.6} by $\psi_n$, summing $n=1,\cdots, k$, integrating on $[0,T]$, then, for all $m\geq k$, we have 
	\begin{equation}\label{03.11.14}
		\int_0^T \lg\pt_{tt}y^m, \psi\rg_{H^{-1}(\Om;w), H_0^1(\Om;w)}\df t+\int_0^T(y^m,\psi)_{H_0^1(\Om;w)}\df t=\int_0^T(f^m,\psi)_{L^2(\Om)}\df t,  
	\end{equation}
	and hence
	\begin{equation}\label{03.11.12}
		\int_0^T\lg\pt_{tt}z, \psi\rg_{H^{-1}(\Om;w), H_0^1(\Om;w)}\df t+\int_0^T(z, \psi)_{H_0^1(\Om;w)}\df t=\int_0^T(f,\psi)_{L^2(\Om)}\df t
	\end{equation}
	by \eqref{03.11.9} and $f^m\ra f$ strongly in $L^2(Q)$. Note that the set of the functions $\psi$ is dense in $L^2(0,T; H_0^1(\Om;w))$, hence \eqref{03.11.12} holds for all $\psi\in L^2(0,T; H_0^1(\Om;w))$. Therefore, 
	\begin{equation*}
		\lg \pt_{tt}z,v\rg_{H^{-1}(\Om;w), H_0^1(\Om;w)}+(z,v)_{H_0^1(\Om;w)}=(f,v)_{L^2(\Om)} 
	\end{equation*}
	for all $v\in H_0^1(\Om;w)$ and a.e.~$t\in [0,T]$.  This proves Definition \ref{03.11.D1} (i). 
	
 We further impose the conditions $\psi (T)=0$ and  $\partial_t \psi (T)=0$. On one hand, integration by parts in \eqref{03.11.12} with respect to $t\in [0,T]$, we get
	\begin{equation*}
		\begin{split} 
		&\int_0^T(z,\pt_{tt}\psi)_{L^2(\Om)}\df t+\int_0^T(z,\psi)_{H_0^1(\Om;w)}\df t\\
		&=\int_0^T(f,\psi)_{L^2(\Om)}\df t-(z(0),\pt_t\psi(0))_{L^2(\Om)}+\lg\pt_tz(0), \psi(0)\rg_{H^{-1}(\Om;w), H_0^1(\Om;w)}. 
		\end{split} 
	\end{equation*}
	On the other hand, 
	integration by   parts in \eqref{03.11.14} with respect to $t\in [0,T]$, we get
	\begin{equation*}
		\begin{split} 
		&\int_0^T(y^m,\pt_{tt}\psi)_{L^2(\Om)}\df t+\int_0^T(y^m, \psi)_{H_0^1(\Om;w)}\df t\\
		&=\int_0^T(f^m,\psi)_{L^2(\Om)}\df t-(y^m(0),\pt_t\psi(0))_{L^2(\Om)}+(\pt_ty^m(0), \psi(0))_{L^2(\Om)}, 
		\end{split} 
	\end{equation*}
	and then 
	\begin{equation*}
		\begin{split}
			&\int_0^T(z,\pt_{tt}\psi)_{L^2(\Om)}\df t+\int_0^T(z,\psi)_{H_0^1(\Om;w)}\df t\\
			&=\int_0^T(f,\psi)_{L^2(\Om)}\df t-(y^0,\pt_t\psi(0))_{L^2(\Om)}+(y^1,\psi(0))_{L^2(\Om)}
		\end{split}
	\end{equation*}
	by \eqref{03.11.9} and \eqref{03.11.5}. Therefore, $z(0)=y^0$ and $\pt_tz(0)=y^1$. This proves Definition \ref{03.11.D1} (ii). 
	
	Overall, we have proved $z$ is a weak solution of the equation \eqref{01.10.1}. 
	
	{\it Step 5}. Uniqueness. 
	
	We only need to show the solution of \eqref{01.10.1} with respect to $(0,0,0)$ is zero function. To verify this, fix $0\leq s\leq T$ and set
	\begin{equation*}
		\xi(t)=
		\begin{cases}
			\int_t^s z(\tau)\df \tau, &\mbox{if } 0\leq t\leq s, \\
			0, &\mbox{if } s\leq t\leq T. 
		\end{cases}
	\end{equation*}
	Then $\xi(t)\in H_0^1(\Om;w)$ for each $t\in [0,T]$, and hence
	\begin{equation*}
		\int_0^s \lg \pt_{tt} z, \xi\rg_{H^{-1}(\Om;w), H_0^1(\Om;w)}\df t+\int_0^s(z, \xi)_{H_0^1(\Om;w)}\df t=0
	\end{equation*}
	by Definition \ref{03.11.D1} (i). 
	Since $\pt_tz(0)=\xi(s)=0$, integrating by parts, we get
	\begin{equation*}
		-\int_0^s (\pt_tz, \pt_t\xi)_{L^2(\Om)}\df t+\int_0^s (z,\xi)_{H_0^1(\Om;w)}\df t=0. 
	\end{equation*}
	Note that $\pt_t\xi=-z\ (0\leq t\leq s)$, and then 
	\begin{equation*}
		\int_0^s( \pt_tz,z)_{L^2(\Om)}\df t-\int_0^T(\pt_t\xi, \xi)_{L^2(\Om)}\df t=0. 
	\end{equation*}
	i.e., 
	\begin{equation*}
		\int_0^s\f{\df}{\df t}\left(\f{1}{2}\|z\|_{L^2(\Om)}^2-\f{1}{2}(\xi,\xi)_{H_0^1(\Om;w)}\right)\df t=0. 
	\end{equation*}
	Hence 
	\begin{equation*}
		\|z(s)\|_{L^2(\Om)}^2+\|\xi(0)\|_{H_0^1(\Om;w)}^2=0. 
	\end{equation*}
	This shows that $z=0$. 
	
	From this we get $z=y$. Together this with  \eqref{03.11.10}, we get the 1th-3th terms in \eqref{03.12.1}. 
	
	{\it Step 6}. Improved regularity. 
	
	Let $y^0\in D(\mcA), y^1\in H_0^1(\Om;w)$ and $f\in H^1(0,T; L^2(\Om))$. From \eqref{03.11.4} we get
	\begin{equation*}
		\f{\df^2}{\df t^2}\f{\df y_n^m}{\df t}+\la_n \f{\df y_n^m}{\df t}=\f{\df f_n^m}{\df t}, \mbox{ for } t\in [0,T], n=1,\cdots, m. 
	\end{equation*}
	Denote $\wt y^m=\pt_ty^m$, then this is 
	\begin{equation}\label{03.12.14}
		(\pt_{tt}\wt y^m, \Phi_n)_{L^2(\Om)}+(\wt y^m, \Phi_n)_{H_0^1(\Om;w)}=(\pt_tf^m, \Phi_n)_{L^2(\Om)}, \ n=1,\cdots, m. 
	\end{equation}
	Multiplying   $\f{\df y_n^m}{\df t}$, summing $n=1,\cdots, m$, we get
	\begin{equation*}
		\begin{split}
			(\pt_{tt}\wt y^m, \pt_t\wt y^m)_{L^2(\Om)}+(\wt y^m, \pt_t\wt y^m)_{H_0^1(\Om;w)}=(\pt_t f^m, \pt_t\wt y^m)_{L^2(\Om)}. 
		\end{split}
	\end{equation*}
	i.e., 
	\begin{equation*}
		\f{\df}{\df t}\left(\|\pt_t\wt y^m\|_{L^2(\Om)}^2+\|\wt y^m\|_{H_0^1(\Om;w)}^2\right)\leq \|\pt_tf^m\|_{L^2(\Om)}^2+\|\pt_t\wt y^m\|_{L^2(\Om)}^2, 
	\end{equation*}
	from the Gronwall's inequality, we get
	\begin{equation}\label{03.12.3}
		\begin{split}
			\|\pt_t\wt y^m\|_{L^2(\Om)}^2+\|\wt y^m\|_{H_0^1(\Om;w)}^2\leq e^t\left(\|\pt_t\wt y^m(0)\|_{L^2(\Om)}^2+\|\wt y^m(0)\|_{H_0^1(\Om;w)}^2+\int_0^t\|\pt_tf^m\|_{L^2(\Om)}^2\df t\right) 
		\end{split}
	\end{equation}
	for all $t\in [0,T]$. 
	
	On one side, from \eqref{03.11.5} and $\wt y^m(0)=\pt_ty^m(0)=\sum_{n=1}^m (y^1,\Phi_n)_{L^2(\Om)}\Phi_n$ and Lemma \ref{06.28.L1}, we obtain 
	\begin{equation}\label{03.12.4}
		\begin{split}
			\|\wt y^m(0)\|_{H_0^1(\Om;w)}^2
			&=\sum_{n=1}^m (y^1,\Phi_n)_{L^2(\Om)}^2\la_n\leq \sum_{n=1}^\iy (y^1,\Phi_n)_{L^2(\Om)}^2\la_n=\|y^1\|_{H_0^1(\Om;w)}^2. 
		\end{split}
	\end{equation}
	On the other side, from \eqref{03.11.5} and \eqref{03.11.6},  we get 
	\begin{equation*}
		\begin{split}
			\pt_t\wt y^m(0)
			&=\pt_{tt}y^m(0)=f^m(0)-\mcA y^m(0)=\sum_{n=1}^m f_n^m(0)\Phi_n-\sum_{n=1}^m y_n^m(0)\Phi_n\la_n\\
			&=\sum_{n=1}^m f_n^m(0)\Phi_n-\sum_{n=1}^m (y^0, \Phi_n)_{L^2(\Om)}\Phi_n\la_n, 
		\end{split}
	\end{equation*}
	and then 
	\begin{equation*}
		\begin{split}
			\|\pt_t\wt y^m(0)\|_{L^2(\Om)}^2
			&\leq 2\sum_{n=1}^m |f_n^m(0)|^2+2\sum_{n=1}^m (y^0,\Phi_n)_{L^2(\Om)}^2\la_n^2\leq C\|f(0)\|_{L^2(\Om)}^2+2\|\mcA y^0\|_{L^2(\Om)}^2
		\end{split}
	\end{equation*}
	by Lemma \ref{06.28.L1}. Together with this and \cite[Theorem 2 (iii) in Chapter 5.9.2, p.~302]{Evans}, we get
	\begin{equation}\label{03.12.5}
		\begin{split}
			\|\pt_t\wt y^m(0)\|_{L^2(\Om)}^2
			&\leq C\|f\|_{H^1(0,T; L^2(\Om))}^2+2\|\mcA y^0\|_{L^2(\Om)}^2, 
		\end{split}
	\end{equation}
	where the positive constant $C$ depending only on $T$. Combining \eqref{03.12.3} and \eqref{03.12.4} and \eqref{03.12.5}, we get
	\begin{equation}\label{03.12.6}
		\begin{split}
			\esssup_{t\in [0,T]}\left(\|\pt_t\wt y^m\|_{L^2(\Om)}^2+\|\wt y^m\|_{H_0^1(\Om;w)}^2\right)\leq C\left(\|y^0\|_{D(\mcA)}^2+\|y^1\|_{H_0^1(\Om;w)}^2+\|f\|_{H^1(0,T; L^2(\Om))}^2\right), 
		\end{split}
	\end{equation}
	where the positive constant $C$ depending only on $T$. Together \eqref{03.12.6} and \eqref{03.11.9}, we get
	\begin{equation*}
		\begin{split}
			\esssup_{t\in [0,T]}\left(\|\pt_{tt}y\|_{L^2(\Om)}^2+\|\pt_{t}y\|_{H_0^1(\Om;w)}^2\right)\leq C\left(\|y^0\|_{D(\mcA)}^2+\|y^1\|_{H_0^1(\Om;w)}^2+\|f\|_{H^1(0,T; L^2(\Om))}^2\right),
		\end{split}
	\end{equation*}
	where the positive constant $C$ depending only on $T$. This proves the third and the  fourth terms in \eqref{03.12.2}. 
	
	Multiplying \eqref{03.11.6} by $\la_ny_n^m$, summing $n=1,\cdots, m$, we get
	\begin{equation*}
		(\pt_{tt}y^m,\mcA y^m)_{L^2(\Om)}+ (\mcA y^m, \mcA y^m)_{L^2(\Om)}=(f^m,\mcA y^m)_{L^2(\Om)},  
	\end{equation*}
	note that $\mcA y^m\in H_0^1(\Om;w)$ by \eqref{03.11.2}, then, from 
	\begin{equation*}
		(\mcA y^m, \mcA y^m)_{L^2(\Om)}=-(\pt_{tt}y^m, \mcA y^m)_{L^2(\Om)}+(f^m, \mcA y^m)_{L^2(\Om)},
	\end{equation*}
	we get
	\begin{equation*}
		\|\mcA y^m\|_{L^2(\Om)}\leq \|f^m\|_{L^2(\Om)}+\|\pt_{tt}y^m\|_{L^2(\Om)}. 
	\end{equation*}
	Together with this and \eqref{03.12.6}, we get 
	\begin{equation}\label{03.12.7}
		\begin{split}
			\|\mcA y^m\|_{L^2(\Om)}\leq C\left(\|y^0\|_{D(\mcA)}^2+\|y^1\|_{H_0^1(\Om;w)}^2+\|f\|_{H^1(0,T; L^2(\Om))}^2\right), 
		\end{split}
	\end{equation}
	where the positive constant $C$ depending only on $T$. 
	
	Choosing $\zeta=\zeta(x_2)\in C_0^\iy(\Om), 0\leq \zeta\leq 1$ such that 
	\begin{equation*}
		\begin{split}
			\zeta=1 \mbox{ on } (-1,1)\ts\left(\f{1}{8},1\right), \quad \zeta=0 \mbox{ on } (-1,1)\ts \left(0,\f{1}{16}\right), \quad |\pt_{x_2}\zeta|\leq C, |\pt_{x_2x_2}\zeta|\leq C, 
		\end{split}
	\end{equation*}
	where the positive constants $C$ are absolute. Then $\xi=\zeta y^m$ is the solution of the following uniformly elliptic equation 
	\begin{equation*}
		\begin{cases}
			-\Div(A\nabla \xi)=\zeta f^m-\zeta \pt_{tt}y^m-2\nabla\zeta \cdot A\nabla y^m-y^m\Div(A\nabla\zeta), &\mbox{in }\Om_\f{1}{32},\\
			\xi=0, &\mbox{on } \pt\Om_\f{1}{32},
		\end{cases}
	\end{equation*}
	and then, from \cite[Theorem 1 (p.~327) in Chapter 6.3.1 and Theorem 4 (p.~334) in Chapter 6.3.2]{Evans} and \eqref{03.11.7} and \eqref{03.12.6}, we get
	\begin{equation}\label{03.12.8}
		\begin{split}
			\|y^m\|_{H^2(\Om_\f{1}{8})}
			&\leq \|\xi\|_{H^2(\Om_\f{1}{16})}\leq C\left(\|f^m\|_{L^2(\Om)}+\|\pt_{tt}y^m\|_{L^2(\Om)}+\|y^m\|_{H_0^1(\Om )}\right)\\
			&\leq C\left(\|y^0\|_{D(\mcA)}+\|y^1\|_{H_0^1(\Om;w)}+\|f\|_{H^1(0,T; L^2(\Om))}\right),  
		\end{split}
	\end{equation} 
	where the positive constant $C$ depending only on $\al$ and $T$. Combining \eqref{03.11.9} and \eqref{03.12.7} and \eqref{03.12.8}, we get the first and the second terms in \eqref{03.12.2}. 
	
	{\it Step 7}. Hidden regularity. i.e., we shall prove the  fourth term in \eqref{03.12.1}. 
	
	By the density argument, we only need to prove the case $y^0\in D(\mcA)$ and $y^1\in H_0^1(\Om;w)$ and $f\in H^1(0,T; L^2(\Om))$.

	Let $\de\in (0,\f{1}{16})$. Let $\zeta_2=\zeta_2(x_2)\in C_0^\iy(\R), 0\leq \zeta_2\leq 1$ such that 
	\begin{equation*}
		\zeta_2=1 \mbox{ on } \left(4\de, +\iy\right), \quad \zeta_2=0 \mbox{ on } (-\iy, 2\de), \quad |\zeta_2'|\leq C\de^{-1}, |\zeta_2''|\leq C\de^{-2}, 
	\end{equation*}
	where the positive constants $C$ are absolute. Taking $\zeta_1=\zeta(x_1)\in C_0^\iy(\R), 0\leq \zeta_1 \leq 1$ such that 
	\begin{equation*}
		\zeta_1=1 \mbox{ on }(-\iy, 0), \quad \zeta_1=0 \mbox{ on } \left(\f{1}{4}, 1\right),\quad |\zeta_1'|\leq C, |\zeta_1''|\leq C, 
	\end{equation*}
	where the positive constants $C$ are absolute. It is clear that $\zeta_1\zeta_2\pt_{x_1}y\in H^1(\Om_\de)$ and $\zeta_1\zeta_2\pt_{x_1}y=0$ on $[(-1,1)\ts (0,\de)]\cup [(\f{1}{2},1)\ts (-1,1)]$ by \eqref{03.12.2} (or Step 6). Multiply $\zeta_1\zeta_2\pt_{x_1}y$ on the both sides of \eqref{01.10.1}, then 
	\begin{equation*}
		\begin{split} 
		A_1+A_2
		&\equiv \iint_Q (\pt_{tt}y)\zeta_1\zeta_2\pt_{x_1}y\df x\df t-\iint_Q [\Div(A\nabla y)]\zeta_1\zeta_2\pt_{x_1}y\df x\df t\\
		&=\iint_Q f\zeta_1\zeta_2\pt_{x_1}y\df x\df t\equiv A_3. 
		\end{split} 
	\end{equation*}
	Since we have ``cut-off" the domain $[(-1,1)\ts (0,\de)]\cup [(\f{1}{2},1)\ts (-1,1)]$, the following integration by parts is feasible. 
	
	Note that from $\pt_ty=0$ on $\pt\Om$, we get 
	\begin{equation}\label{03.12.9}
		\begin{split}
			A_1
			&= \int_\Om (\pt_ty)(\zeta_1\zeta_2\pt_{x_1}y)\df x\bigg|_{t=0}^{t=T}-\f{1}{2}\iint_Q \zeta_1\zeta_2\pt_{x_1}(\pt_ty)^2\df x\df t\\
			&=\int_\Om (\pt_ty)(\zeta_1\zeta_2\pt_{x_1}y)\df x\bigg|_{t=0}^{t=T}+\f{1}{2}\iint_Q \zeta_1'\zeta_2(\pt_ty)^2\df x\df t. 
		\end{split}
	\end{equation}
	
	From $y=0$ on $(-1,1)\ts \{1\}$, then $\pt_{x_1}y=0$ on $(-1,1)\ts \{1\}\ts (0,T)$, and hence  $\zeta_1\zeta_2\pt_{x_1}y=0$ on $(\pt\Om-[\{-1\}\ts (0,1)])\ts (0,T)$. From $y=0$ on $\{-1\}\ts (0,1)$, then $\pt_{x_2}y=0$ on $\{-1\}\ts (0,1)\ts (0,T)$.  Therefore, 
	\begin{equation}\label{03.12.10}
		\begin{split}
			A_2
			&=\f{1}{2}\iint_{[\{-1\}\ts (0,1)]\ts (0,T)}\zeta_1\zeta_2(\pt_{x_1}y)^2\df S\df t\\
			&\hspace{4.5mm}+\iint_{Q} \zeta_1'\zeta_2(\pt_{x_1}y)^2\df x\df t+\iint_Q \zeta_1\zeta_2'x_2^\al (\pt_{x_2}y)\pt_{x_1}y\df x\df t\\
			&\hspace{4.5mm}-\f{1}{2}\iint_Q \zeta_1'\zeta_2(\pt_{x_1}y)^2\df x\df t-\f{1}{2}\iint_Q \zeta_1'\zeta_2x_2^\al (\pt_{x_2}y)^2\df x\df t. 
		\end{split}
	\end{equation}
	
	It is obvious that 
	\begin{equation}\label{03.12.11}
		\begin{split}
			A_3
			&\leq \f{1}{2}\iint_Q \zeta_1\zeta_2f^2\df x\df t+\f{1}{2}\iint_Q\zeta_1\zeta_2(\pt_{x_1}y)^2\df x\df t.  
		\end{split}
	\end{equation}
	
	Combining \eqref{03.12.9}, \eqref{03.12.10} and \eqref{03.12.11}, we get
	\begin{equation*}
		\begin{split}
			&\f{1}{2}\iint_{[\{-1\}\ts (0,1)]\ts (0,T)}\zeta_1\zeta_2(\pt_{x_1}y)^2\df S\df t\\
			&\leq \f{1}{2}\|f\|_{L^2(Q)}^2+C\iint_Q \left[(\pt_ty)^2+  \nabla y\cdot A\nabla y\right]\df x\df t+C\esssup_{t\in[0,T]}\left(\|\pt_ty\|_{L^2(\Om)}^2+\|y\|_{H_0^1(\Om;w)}^2\right)\\
			&\leq C\left(\|y^0\|_{H_0^1(\Om;w)}^2+\|y^1\|_{L^2(\Om)}^2+\|f\|_{L^2(Q)}^2\right) 
		\end{split}
	\end{equation*}
	by \eqref{03.12.1}.
	This implies 
	\begin{equation}\label{03.12.12}
		\begin{split}
			\iint_{[\{-1\}\ts (4\de,1)]\ts (0,T)}(\pt_{x_1}y)^2\df S\df t\leq C\left(\|y^0\|_{H_0^1(\Om;w)}^2+\|y^1\|_{L^2(\Om)}^2+\|f\|_{L^2(Q)}^2\right), 
		\end{split}
	\end{equation}
	where the positive constants $C$ depending only on $\al, T$ and $\de$. Similar to this case, we get 
	\begin{equation}\label{03.16.1}
		\begin{split}
			\iint_{[\{1\}\ts (4\de,1)]\ts (0,T)}(\pt_{x_1}y)^2\df S\df t
			&\leq C\left(\|y^0\|_{H_0^1(\Om;w)}^2+\|y^1\|_{L^2(\Om)}^2+\|f\|_{L^2(Q)}^2\right),\\
			\iint_{[(-1,1)\ts\{1\}]\ts (0,T)}(\pt_{x_2}y)^2\df S\df t
			&\leq C\left(\|y^0\|_{H_0^1(\Om;w)}^2+\|y^1\|_{L^2(\Om)}^2+\|f\|_{L^2(Q)}^2\right), 
		\end{split}
	\end{equation}
	where the positive constants $C$ depending only on $\al, T$ and $\de$. Together with these  and  \eqref{03.12.12}, we get the fourth term in \eqref{03.12.1}. 
	
	{\it Step 8}. Hidden regularity. i.e., we shall prove the fifth term in \eqref{03.12.2}. 
	
	It is clear that $h=\pt_ty$ is the weak solution of the following system 
	\begin{equation}\label{03.12.13}
		\begin{cases}
			\pt_{tt} h-\Div(A\nabla h)=\pt_tf, &\mbox{in }Q, \\
			h=0, &\mbox{on }\pt Q, \\
			h(0)=y^1, \pt_{t}h(0)=f(0)-\mcA y^0, &\mbox{in }\Om. 
		\end{cases}
	\end{equation}
	Indeed, from \eqref{03.12.14} and \eqref{03.11.5} and \eqref{03.11.9}, we get it. Then, from Step 7 (i.e., the fourth term in \eqref{03.12.1}), we get the fifth term in \eqref{03.12.2}.   We complete the proof of this theorem.  
\end{proof}

\begin{remark}\label{03.14.R1}
	Let $\de\in (0,\f{1}{8})$. In the proof of \eqref{03.12.8}, if we replace $\zeta\in C_0^\iy(\Om), 0\leq \zeta\leq 1$ such that 
	\begin{equation*}
		\zeta=1 \mbox{ on } (-1,1)\ts (2\de,1),\quad \zeta=0 \mbox{ on } (-1,1)\ts (-\iy, \de), \quad |\pt_{x_2}\zeta|\leq C\de^{-1}, |\pt_{x_2x_2}\zeta|\leq C\de^{-2}, 
	\end{equation*}
	then we have the following estimate
	\begin{equation*}
		\esssup_{t\in[0,T]}\|y\|_{H^2(\Om_\de)}\leq C_\de\left(\|y^0\|_{D(\mcA)}+\|y^1\|_{H_0^1(\Om;w)}+\|f\|_{H^1(0,T; L^2(\Om))}\right), 
	\end{equation*}
	where the positive constant $C_\de$ depending only on $\al$ and $T$ and $\de$. 
\end{remark}

\begin{proposition}\label{03.13.P1}
	Let $y^0\in H_0^1(\Om;w)$ and $y^1\in L^2(\Om)$ and $f\in L^2(Q)$. Then the function 
	\begin{equation*}
		y\in L^2(0,T; H_0^1(\Om;w))\cap H^1(0,T; L^2(\Om))\cap H^2(0,T; H^{-1}(\Om;w))
	\end{equation*}
	is a weak solution of \eqref{01.10.1} if and only if the function 
	\begin{equation*}
		y\in L^2(0,T; H_0^1(\Om;w))\cap H^1(0,T; L^2(\Om))
	\end{equation*}
	satisfies 
	\begin{equation}\label{03.13.1}
		\begin{split}
			&\int_0^T (y, \pt_{tt}\psi)_{L^2(\Om)}\df t+\int_0^T(y, \psi)_{H_0^1(\Om;w)}\df t\\
			&=\int_0^T(f,\psi)_{L^2(\Om)}\df t-(y^0, \pt_t\psi(0))_{L^2(\Om)}+(y^1,\psi(0))_{L^2(\Om)}, 
		\end{split}
	\end{equation}
	where $\psi\in C^\iy(\ol Q)$ with $\supp\psi(t)\s \Om$ for all $t\in [0,T]$ and $\psi(T)=\pt_t\psi(T)=0$. 
\end{proposition}

\begin{proof}
	We prove the necessity. 
	
	Let $y$ be the weak solution of \eqref{01.10.1} with respect to $(y^0,y^1,f)$. Then 
	\begin{equation}\label{03.13.2}
		y\in C([0,T]; L^2(\Om)), \quad \pt_ty\in C([0,T]; H^{-1}(\Om;w)). 
	\end{equation}
	Let $\psi\in C^\iy(\ol Q)$ with $\supp \psi(t)\s \Om$ for all $t\in [0,T]$ and $\psi(T)=\pt_t\psi(T)=0$. Then, from Definition \ref{03.11.D1} (i), we get
	\begin{equation*}
		\begin{split}
			\lg\pt_{tt}y,\psi(t)\rg_{H^{-1}(\Om;w), H_0^1(\Om;w)}+(y, \psi(t))_{H_0^1(\Om;w)}=(f,\psi(t))_{L^2(\Om)}. 
		\end{split}
	\end{equation*} 
	Integrating on $[0,T]$, we obtain 
	\begin{equation*}
		\begin{split}
			\int_0^T\lg \pt_{tt}y,\psi\rg_{H^{-1}(\Om;w), H_0^1(\Om;w)}\df t+\int_0^T(y,\psi)_{H_0^1(\Om;w)}\df t=\int_0^T(f,\psi)_{L^2(\Om)}\df t. 
		\end{split}
	\end{equation*}
	Integration by parts with respect to $t\in [0,T]$, then 
	\begin{equation*}
		\begin{split}
			&\int_0^T(y,\pt_{tt}\psi)_{L^2(\Om)}\df t+\int_0^T(y,\psi)_{H_0^1(\Om;w)}\df t\\
			&=\int_0^T(f,\psi)_{L^2(\Om)}\df t-(y^0,\pt_t\psi(0))_{L^2(\Om)}+\lg y^1,\psi(T)\rg_{H^{-1}(\Om;w), H_0^1(\Om;w)}\\
			&=\int_0^T(f,\psi)_{L^2(\Om)}\df t-(y^0,\pt_t\psi(0))_{L^2(\Om)}+(y^1,   \psi(0))_{L^2(\Om)}. 
		\end{split}
	\end{equation*}
	This shows \eqref{03.13.1}. 
	
	We prove sufficiency. 
	
	Note that 
	\begin{equation*}
		\begin{split} 
		\mcA: L^2(0,T; H_0^1(\Om;w))\ra L^2(0,T; H^{-1}(\Om;w))
		\end{split} 
	\end{equation*}
	is defined by 
	\begin{equation*}
		\lg \mcA z, v\rg_{L^2(0,T; H^{-1}(\Om;w)), L^2(0,T; H_0^1(\Om;w))}=\int_0^T(z,v)_{H_0^1(\Om;w)}\df t, 
	\end{equation*}
	then 
	\begin{equation}\label{03.13.3}
		\|\mcA z\|_{L^2(0,T; H^{-1}(\Om;w))}\leq \|z\|_{L^2(0,T; H_0^1(\Om;w))}. 
	\end{equation}

	Let $\psi\in C_0^\iy(Q)$. Then, from \eqref{03.13.1}, we get
	\begin{equation*}
		\begin{split}
			\int_0^T(y,\pt_{tt}\psi)_{L^2(\Om)}\df t+\int_0^T(y,\psi)_{H_0^1(\Om;w)}\df t=\int_0^T(f,\psi)_{L^2(\Om)}\df t. 
		\end{split}
	\end{equation*}
	Together with this and $f\in L^2(Q)\s L^2(0,T; H^{-1}(\Om;w))$  and  \eqref{03.13.3}, we get
	\begin{equation*}
		\begin{split}
			\pt_{tt}y=f-\mcA y \mbox{ in the sense of distribution}, 
		\end{split}
	\end{equation*} 
	and $\pt_{tt}y\in L^2(0,T; H^{-1}(\Om;w))$. 
	
	Let $t\in (0,T)$. Set $\e\in (0,\f{1}{2}\min \{t, T-t\})$, and $\de\in (0,\f{1}{2}\e)$. Choosing $\zeta\in C^\iy(\ol\R), 0\leq \zeta\leq 1$ such that 
	\begin{equation*}
		\zeta=1 \mbox{ on } (t-\e,t+\e), \quad \zeta=0 \mbox{ on }(-\iy, t-\e-\de)\cup (t+\e+\de,+\iy). 
	\end{equation*}
	For each $v\in C_0^\iy(\Om)$, we have $\zeta v\in C_0^\iy(Q)$. And from \eqref{03.13.1} we obtain 
	\begin{equation*}
		\begin{split}
			\int_0^T\zeta \lg \pt_{tt}y, v\rg_{H^{-1}(\Om;w), H_0^1(\Om;w)}\df t+\int_0^T\zeta  (y, v)_{H_0^1(\Om;w)}\df t=\int_0^T\zeta (f,v)_{L^2(\Om)}\df t. 
		\end{split}
	\end{equation*}
	Letting $\de\ra 0^+$, we get
	\begin{equation*}
		\int_{t-\e}^{t+\e} \lg \pt_{tt}y, v\rg_{H^{-1}(\Om;w), H_0^1(\Om;w)}\df t+\int_{t-\e}^{t+\e}  (y, v)_{H_0^1(\Om;w)}\df t=\int_{t-\e}^{t+\e}  (f,v)_{L^2(\Om)}\df t. 
	\end{equation*}
	From the Lebesgue point theorem, letting $\e\ra 0^+$, we get
	\begin{equation*}
		\begin{split}
			\lg\pt_{tt}y, v\rg_{H^{-1}(\Om;w), H_0^1(\Om;w)}+(y,v)_{H_0^1(\Om;w)}=(f,v)_{L^2(\Om)}
		\end{split}
	\end{equation*}
	for a.e.~$t\in [0,T]$ and all $v\in C_0^\iy(\Om)$. This proves Definition \ref{03.11.D1} (i). 
	
	Let $v\in C_0^\iy(\Om)$ and $\zeta\in C^\iy(\ol\R)$ such that $\zeta(T)=\pt_t\zeta(T)=0$, taking $\psi=\zeta v$, then, from \eqref{03.13.1}, we get
	\begin{equation*}
		\begin{split}
			\int_0^T(y,v)_{L^2(\Om)}\zeta''\df t+\int_0^T (y,v)_{H_0^1(\Om;w)}\zeta\df t=\int_0^T(f,v)_{L^2(\Om)}\zeta\df t-(y^0,v)_{L^2(\Om)}\zeta'(0)+(y^1,v)_{L^2(\Om)}\zeta(0). 
		\end{split}
	\end{equation*}
	Denote $y_n=(y,\Phi_n)_{L^2(\Om)}$ for all $n\in\N$, then $y=\sum\limits_{n=1}^\iy y_n\Phi_n$.   By density, it suffices to take  $v=  \Phi_n\ (n\in\N)$, then 
	\begin{equation*}
		\int_0^Ty_n\zeta''\df t+\la_n\int_0^T y_n\zeta \df t=\int_0^T f_n\zeta\df t-y_n^0\zeta'(0)+y_n^1\zeta(0), 
	\end{equation*}
	where $f_n=(f,\Phi_n)_{L^2(\Om)}, y_n^0=(y^0,\Phi_n)_{L^2(\Om)}$ and $y_n^1=(y^1,\Phi_n)_{L^2(\Om)}$. It is obviously that $y_n$ is the solution of the following ordinary differential equation 
	\begin{equation*}
		\begin{cases} 
		y_n''+\la_n y_n=f_n, &\mbox{for }t\in [0,T], \\
		y_n(t)=y_n^0, y_n'(t)=y_n^1, &\mbox{for } t=0.  
		\end{cases} 
	\end{equation*}
	Note that $y^m=\sum_{n=1}^m y_n\Phi_n$ is defined in \eqref{03.11.4} and \eqref{03.11.5}, from the proof of Theorem \ref{03.11.T1} (Steps 1-3), we know that $y^m\ra \wt y$ weakly in $L^2(0,T; H_0^1(\Om;w))$ as $m\ra\iy$, where $\wt y$ is the weak solution of \eqref{01.10.1} with respect to $(y^0,y^1,f)$.  It is clearly that $y^m\ra y$ strongly in $L^2(0,T; H_0^1(\Om;w))$, then $\wt y=y$. Moreover, we have $y(0)=\wt y(0)=y^0$ and $\pt_ty(0)=\pt_t\wt y(0)=y^1$. This shows Definition \ref{03.11.D1} (ii). We complete the proof of this proposition. 
\end{proof}

\section{Shape design}\label{S3}

Let $\de\in (0,\f{1}{4})$. Denote 
\begin{equation*}
	w_\de=w|_{\Om_\de},\quad A_\de=A|_{\Om_\de}, \mbox{ and } \mcA_\de u=-\Div(A_\de\nabla u) \mbox{ on }\Om_\de. 
\end{equation*}

Now, we define
\begin{equation*}
	H^1(\Om_\de;w_\de)=\left\{u\in L^2(\Om_\de)\colon \int_{\Om_\de} \nabla u\cdot A_\de\nabla u\df x<+\iy \right\}, 
\end{equation*}
its inner product and norm are defined by 
\begin{equation*}
	(u,v)_{H^1(\Om_\de;w_\de)}=\int_{\Om_\de}\left(uv+\nabla u\cdot A_\de\nabla v\right)\df x,\quad \|u\|_{H^1(\Om_\de;w_\de)}=(u,u)_{H^1(\Om_\de;w_\de)}^\f{1}{2}. 
\end{equation*}
Denote 
\begin{equation*}
	H_0^1(\Om_\de;w_\de)=\mbox{the closure of } C_0^\iy(\Om_\de) \mbox{ in } H^1(\Om_\de;w_\de), 
\end{equation*}
and $H^{-1}(\Om_\de;w_\de)$ the dual space of $H_0^1(\Om_\de;w_\de)$ with pivot $L^2(\Om_\de)$. 

Define
\begin{equation*}
	H^2(\Om_\de;w_\de)=\left\{u\in H^1(\Om_\de;w_\de)\colon \mcA_\de u\in L^2(\Om_\de)\right\}, 
\end{equation*}
its inner product and norm are defined by 
\begin{equation*}
	(u,v)_{H^2(\Om_\de;w_\de)}=(u,v)_{H^1(\Om_\de;w_\de)}+\int_{\Om_\de}(\mcA_\de u)(\mcA_\de v)\df x, \quad \|u\|_{H^2(\Om_\de;w_\de)}=(u,u)_{H^2(\Om_\de;w_\de)}^\f{1}{2}. 
\end{equation*}

It is obviously that the spaces 
\begin{equation*}
	H_0^1(\Om_\de;w_\de),\quad H^1(\Om_\de;w_\de), \quad H^2(\Om_\de;w_\de) 
\end{equation*}
are Hilbert spaces, since $H_0^1(\Om_\de;w_\de)=H_0^1(\Om_\de), H^1(\Om_\de;w_\de)=H^1(\Om_\de)$, and $H^2(\Om_\de;w_\de)=H^2(\Om_\de)$ by the regularity of uniformly elliptic equation. 

\begin{lemma}\label{03.13.L1}
	Let $\al\in (0,1)$. For all $u\in H_0^1(\Om_\de;w_\de)$, we have
	\begin{equation*}
		\int_{\Om_\de} x_2^{\al-2}u^2\df x\leq \f{4}{(1-\al)^2}\int_{\Om_\de} x_2^\al (\pt_{x_2}u)^2\df x. 
	\end{equation*}
\end{lemma}

\begin{proof}
	The proof is the same as that of Lemma \ref{03.11.L1}, since $H_0^1(\Omega_\delta;w_\delta) \subset H_0^1(\Omega;w)$.
\end{proof}

\begin{remark}\label{03.13.R1}
	Let $\al\in (0,1)$. Then for all $u\in H_0^1(\Om_\de;w_\de)$ we have 
	\begin{equation}\label{03.13.5}
		\int_{\Om_\de} u^2\df x=\int_{\Om_\de} x_2^{2-\al}x_2^{\al-2}u^2\df x\leq \f{4}{(1-\al)^2}\int_{\Om_\de} x_2^\al (\pt_{x_2}u)^2\df x. 
	\end{equation}
	Hence the inner product and the norm are defined by 
	\begin{equation*}
		(u,v)_{H_0^1(\Om_\de;w_\de)}=\int_{\Om_\de} A_\de\nabla u\cdot \nabla u\df x,\quad \|u\|_{H_0^1(\Om_\de;w_\de)}=(u,u)_{H^1(\Om_\de;w_\de)}^\f{1}{2}
	\end{equation*}
	are equivalent inner product and norm on $H_0^1(\Om_\de;w_\de)$, respectively. 
\end{remark}

\begin{lemma}\label{03.13.L2}
	The embedding $H_0^1(\Om_\de;w_\de)\hra L^2(\Om_\de)$ is compact. 
\end{lemma}

\begin{proof}
	From the embedding  $H_0^1(\Om_\de;w_\de)=H_0^1(\Om_\de)\hra L^2(\Om_\de)$ is compact, we get this lemma. 
\end{proof}

From Remark \ref{03.13.R1}, Lemma \ref{03.13.L2}  and Lemma 3.1 of \cite{YangGuo1}, we obtain that there exists a discrete spectrum for the partial differential operator $\mcA_\de$: 
\begin{equation*}
	0< \la_1^\de<\la_2^\de\leq \la_3^\de\leq \cdots \ra +\iy. 
\end{equation*}
Note that 
\begin{equation*}
	\f{(1-\al)^2}{4}\leq \la_1^\de=\inf_{0\neq u\in H_0^1(\Om_\de;w_\de)}\f{\int_{\Om_\de} A_\de\nabla u\cdot \nabla u \df x}{\int_{\Om_\de} u^2\df x}. 
\end{equation*}

\begin{notation}\label{03.13.N1}
	Let $\Phi_n^\de\ (n\in\N)$ be the eigenfunction of $\mcA_\de$ with respect to the eigenvalue $\la_n^\de$. i.e., 
	\begin{equation}\label{03.13.6}
		\begin{cases}
			\mcA_\de\Phi_n^\de=\la_n^\de\Phi_n^\de, &\mbox{in }\Om_\de,\\
			\Phi_n^\de=0, &\mbox{on }\pt\Om_\de. 
		\end{cases}
	\end{equation}
	Moreover, set $\|\Phi_n^\de\|_{L^2(\Om_\de)}=1$ for all $n\in\N$, then  $\{\Phi_n^\de\}_{n\in\N}$ is a orthonormal basis of $L^2(\Om_\de)$, and  $\{\Phi_n^\de\}_{n\in\N}$ is the orthonormal basis of $L^2(\Om_\de)$, moreover, $\{\Phi_n^\de\}_{n\in\N}$ is an orthogonal subset of $H_0^1(\Om_\de;w_\de)$. (See \cite[Theorem 7 (pp. 728) in Appendix D]{Evans}, or the proof of Lemma \ref{03.13.L3} in the following.) 
\end{notation}

\begin{lemma}\label{03.13.L3}
	Let $u=\sum\limits_{i=1}^\iy u_i \Phi_i^\de\in H_0^1(\Om_\de;w_\de)$ with $u_i=(u,\Phi_i^\de)_{L^2(\Om_\de)}$ for all $i\in\N$. We have $\nabla u=\sum\limits_{i=1}^\iy u_i\nabla \Phi_i^\de$ and $\|u\|_{H_0^1(\Om_\de;w_\de)}=(\sum\limits_{i=1}^\iy u_i^2\la_i^\de)^\f{1}{2}$, and
	\begin{equation*}
		u\in H^2(\Om_\de;w_\de)\Lra \sum_{i=1}^\iy u_i^2(\la_i^\de)^2<\iy,
	\end{equation*}
	and
	\begin{equation*}
		\mcA_\de u=\sum_{i=1}^\iy u_i\la_i^\de\Phi_i^\de \mbox{ and } \|\mcA_\de u\|_{L^2(\Om_\de)}=\left(\sum_{i=1}^\iy u_i^2(\la_i^\de)^2\right)^\f{1}{2}.
	\end{equation*}
\end{lemma}

\begin{proof}
	The proof follows from the same argument as in the proof of Lemma \ref{06.28.L1}.
\end{proof}

We consider the following equation 
\begin{equation}\label{03.13.4}
	\begin{cases}
		\pt_{tt}y_\de-\Div(A_\de\nabla y_\de)=f_\de, &\mbox{in }Q_\de,\\
		y_\de=0, &\mbox{on }\pt Q_\de,\\
		y_\de(0)=y_\de^0,\pt_ty_\de(0)=y_\de^1, &\mbox{in }\Om_\de, 
	\end{cases}
\end{equation}
where $y_\de^0\in H_0^1(\Om_\de;w_\de), y_\de^1\in L^2(\Om_\de)$ and $f_\de\in L^2(Q_\de)$. 

\begin{definition}\label{03.13.D1}
	Let $y_\de^0\in H_0^1(\Om_\de;w_\de), y_\de^1\in L^2(\Om_\de)$ and $f_\de\in L^2(Q_\de)$. We call 
	\begin{equation*}
		y_\de\in L^2(0,T; H_0^1(\Om_\de;w_\de))\cap H^1(0,T; L^2(\Om_\de))\cap H^2(0,T; H^{-1}(\Om_\de;w_\de))
	\end{equation*}
	is a weak solution of  \eqref{03.13.4} with respect to $(y_\de^0,y_\de^1,f_\de)$, provided 
	
	(i) for each $v\in H_0^1(\Om_\de;w_\de)$ and a.e.~$t\in [0,T]$, we have
	\begin{equation*}
		\lg\pt_{tt}y_\de, v\rg_{H^{-1}(\Om_\de;w_\de), H_0^1(\Om_\de;w_\de)}+(y_\de, v)_{H_0^1(\Om_\de;w_\de)}=(f_\de,v)_{L^2(\Om_\de)}, 
	\end{equation*}
	
	(ii) $y_\de(0)=y_\de^0$ and $\pt_ty_\de(0)=y_\de^1$. 
\end{definition}

\begin{theorem}\label{03.13.T1}
	Let $y_\de^0\in H_0^1(\Om_\de;w_\de)$, $y_\de^1\in L^2(\Om_\de)$ and $f_\de\in L^2(Q_\de)$. Then there exists a unique weak solution 
	\begin{equation*}
		y_\de\in L^2(0,T; H_0^1(\Om_\de;w_\de))\cap H^1(0,T; L^2(\Om_\de))\cap H^2(0,T; H^{-1}(\Om_\de;w_\de))
	\end{equation*}
	for the equation \eqref{01.10.1} with respect to $(y_\de^0,y_\de^1, f_\de)$, and 
	\begin{equation}\label{03.13.7}
		\begin{split}
			&\esssup_{t\in [0,T]}\left(\|y_\de(t)\|_{H_0^1(\Om_\de;w_\de)}+\|\pt_ty_\de\|_{L^2(\Om_\de)}\right)+\|\pt_{tt}y_\de\|_{L^2(0,T; H^{-1}(\Om_\de;w_\de))}\\
			&\hspace{4.5mm}+\left\|\f{\pt y_\de}{\pt \nu}\right\|_{L^2(0,T; L^2(\Ga_\de))}\\
			&\leq C\left(\|y_\de^0\|_{H_0^1(\Om_\de;w_\de)}+\|y_\de^1\|_{L^2(\Om_\de)}+\|f\|_{L^2(Q_\de)}\right), 
		\end{split}
	\end{equation}
	where the positive constant $C$ depending only on $\al$ and $T$. 
	
	Furthermore, if in addition $y_\de^0\in D(\mcA_\de)$, and $y_\de^1\in H_0^1(\Om_\de;w_\de)$ and $f_\de\in H^1(0,T; L^2(\Om_\de))$, then 
	\begin{equation}\label{03.13.8}
		\begin{split}
			&\esssup_{t\in [0,T]}\left(\|y_\de\|_{H^2(\Om_\f{1}{8})}+\|y_\de\|_{D(\mcA_\de)}+\|\pt_ty_\de\|_{H_0^1(\Om_\de;w_\de)}+\|\pt_{tt}y_\de\|_{L^2(\Om_\de)}\right)\\
			&\hspace{4.5mm}+\left\|\f{\pt (\pt_ty_\de)}{\pt \nu}\right\|_{L^2(0,T; L^2(\Ga_\de))} \\
			&\leq C\left(\|y_\de^0\|_{D(\mcA_\de)}+\|y_\de^1\|_{H_0^1(\Om_\de;w_\de)}+\|f\|_{H^1(0,T; L^2(\Om_\de))}\right), 
		\end{split}
	\end{equation}
	where the positive constant $C$ depending only on $\al$ and $T$. 
\end{theorem}

\begin{proof}
	The proof of this theorem is similar to that of Theorem \ref{03.11.T1}.
	
	More precisely, \eqref{03.13.7} and \eqref{03.13.8} correspond to \eqref{03.12.1} and \eqref{03.12.2}, respectively. To derive \eqref{03.12.1} and \eqref{03.12.2}, we used the estimates \eqref{03.11.7}, \eqref{03.11.8}, \eqref{03.11.10}, \eqref{03.12.3}, \eqref{03.12.6}, \eqref{03.12.8}, and \eqref{03.12.12}. These estimates rely on Gronwall’s inequality, Remark \ref{03.11.R1}, Lemma \ref{06.28.L1}, \cite[Theorem 2 (iii), Chapter 5.9.2, p.~302]{Evans}, and \cite[Theorem 1, Chapter 6.3.1, p.~327; Theorem 4, Chapter 6.3.2, p.~334]{Evans}.
	
	By replacing Remark \ref{03.11.R1} with Remark \ref{03.13.R1} and Lemma \ref{06.28.L1} with Lemma \ref{03.13.L3}, while noting that the other arguments are independent of $\delta \in (0,\tfrac{1}{4})$, we obtain the corresponding results \eqref{03.13.7} and \eqref{03.13.8}.
	
	We now give a more detailed computation of the fourth term in \eqref{03.13.7}. 
	
	Choosing $\zeta_1=\zeta_1(x_1)\in C_0^\iy(\R), 0\leq \zeta_1\leq 1$ such that 
	\begin{equation*}
		\zeta_1=1 \mbox{ on } \left(-\iy, -\f{1}{2}\right), \quad \zeta_1=0 \mbox{ on } (0,+\iy), \quad |\zeta_1'|\leq C, |\zeta_1''|\leq C, 
	\end{equation*}
	where the positive constants $C$ are absolute. 
	
	Multiplying $\zeta_1\pt_{x_1}y_\de$ on the both sides of \eqref{03.13.4}, integrating on $Q_\de$,  we get
	\begin{equation}\label{03.13.9}
		\begin{split}
			A_1+A_2
			&\equiv \iint_{Q_\de} (\pt_{tt}y_\de)(\zeta_1\pt_{x_1}y_\de)\df x\df t-\iint_{Q_\de} [\Div(A_\de\nabla y_\de)](\zeta_1\pt_{x_1}y_\de)\df x\df t\\
			&=\iint_{Q_\de} f_\de (\zeta_1\pt_{x_1}y_\de)\df x\df t\equiv A_3. 
		\end{split}
	\end{equation}
	
	From $\pt_ty_\de=0$ on $\pt\Om_\de$, then 
	\begin{equation}\label{03.13.10}
		\begin{split}
			A_1
			&=\int_{\Om_\de}\zeta_1(\pt_ty_\de)(\pt_{x_1}y_\de)\df x\bigg|_{t=0}^{t=T}+\f{1}{2}\iint_{Q_\de}\zeta_1' (\pt_ty_\de)^2\df x\df t. 
		\end{split}
	\end{equation}
	From $\pt_{x_1}y_\de=0$ on $\pt\Om_\de\cap \{x\in\R^2\colon x_2=\de,1\}$ and $\pt_{x_2}y_\de=0$ on $\pt\Om_\de\cap\{x\in\R^2\colon x_1=-1\}$, we get
	\begin{equation}\label{03.13.11}
		\begin{split}
			A_2
			&=\f{1}{2}\iint_{[\{-1\}\ts (0,1)]\ts (0,T)}(\pt_{x_1}y_\de)^2\df S\df t\\
			&\hspace{4.5mm}+\f{1}{2}\iint_{Q_\de}\zeta_1'(\pt_{x_1}y_\de)^2\df x\df t-\f{1}{2}\iint_{Q_\de}\zeta_1'x_2^\al (\pt_{x_2}y_\de)^2\df x\df t. 
		\end{split}
	\end{equation}
	And we have 
	\begin{equation}\label{03.13.12}
		\begin{split}
			|A_3|\leq \f{1}{2}\iint_{Q_\de} \zeta_1f_\de^2\df x\df t+\f{1}{2}\iint_{Q_\de} \zeta_1(\pt_{x_1}y_\de)^2\df x\df t. 
		\end{split}
	\end{equation}
	Combining \eqref{03.13.9}-\eqref{03.13.12}, we get
	\begin{equation}\label{03.13.13}
		\begin{split}
			&\iint_{[\{-1\}\ts (0,1)]\ts (0,T)}(\pt_{x_1}y_\de)^2\df S\df t\\
			&\leq \esssup_{t\in[0,T]}\left(\|\pt_ty_\de\|_{L^2(\Om)}^2+\|y_\de\|_{H_0^1(\Om;w)}^2\right)+C\iint_{Q_\de}\left[(\pt_ty_\de)^2+A\nabla y_\de\cdot \nabla y_\de\right)\df x\df t+\|f_\de\|_{L^2(Q)}^2\\
			&\leq C\left(\|y_\de^0\|_{H_0^1(\Om_\de;w_\de)}^2+\|y_\de^1\|_{L^2(\Om_\de)}^2+\|f_\de\|_{L^2(Q_\de)}^2\right)
		\end{split}
	\end{equation}
	by the first and the second terms in \eqref{03.13.7}, where the positive constant $C$ depending  only on $\al$ and $T$. 
	
	By the same argument as \eqref{03.13.13}, we get
	\begin{equation}\label{03.13.14}
		\begin{split}
			\iint_{[\{1\}\ts (0,1)]\ts (0,T)}(\pt_{x_1}y_\de)^2\df S\df t
			&\leq C\left(\|y_\de^0\|_{H_0^1(\Om_\de;w_\de)}^2+\|y_\de^1\|_{L^2(\Om_\de)}^2+\|f_\de\|_{L^2(Q_\de)}^2\right),\\
			\iint_{[(-1,1)\ts \{1\}]\ts (0,T)}(\pt_{x_2}y_\de)^2\df S\df t
			&\leq C\left(\|y_\de^0\|_{H_0^1(\Om_\de;w_\de)}^2+\|y_\de^1\|_{L^2(\Om_\de)}^2+\|f_\de\|_{L^2(Q_\de)}^2\right), 
		\end{split}
	\end{equation}
	where the positive constant $C$ depending  only on $\al$ and $T$.  This complete the proof of this theorem. 
\end{proof}

\begin{proposition}\label{03.14.P1}
	Let $y_\de^0\in H_0^1(\Om_\de;w_\de)$ and $y_\de^1\in L^2(\Om_\de)$ and $f_\de\in L^2(Q_\de)$. Then the function 
	\begin{equation*}
		y_\de\in L^2(0,T; H_0^1(\Om_\de;w_\de))\cap H^1(0,T; L^2(\Om_\de))\cap H^2(0,T; H^{-1}(\Om_\de;w_\de))
	\end{equation*}
	is a weak solution of \eqref{03.13.4} if and only if the function 
	\begin{equation*}
		y_\de\in L^2(0,T; H_0^1(\Om_\de;w_\de))\cap H^1(0,T; L^2(\Om_\de))
	\end{equation*}
	satisfies 
	\begin{equation}\label{03.14.1}
		\begin{split}
			&\int_0^T (y_\de, \pt_{tt}\psi)_{L^2(\Om_\de)}\df t+\int_0^T(y_\de, \psi)_{H_0^1(\Om_\de;w_\de)}\df t\\
			&=\int_0^T(f_\de,\psi)_{L^2(\Om_\de)}\df t-(y_\de^0, \pt_t\psi(0))_{L^2(\Om_\de)}+(y_\de^1,\psi(0))_{L^2(\Om_\de)}, 
		\end{split}
	\end{equation}
	where $\psi\in C^\iy(\ol Q_\de)$ with $\supp\psi(t)\s \Om_\de$ for all $t\in [0,T]$ and $\psi(T)=\pt_t\psi(T)=0$. 
\end{proposition}

\begin{proof}
	The proof of this proposition is the same as that of Proposition \ref{03.13.P1}.
\end{proof}

We now introduce the shape design method.

\begin{notation}\label{03.13.N2}
	Let $y_\de\in L^2(Q_\de)$. Extending $y_\de$ to 
	\begin{equation*}
		Ey_\de(x,t)=
		\begin{cases}
			y_\de(x,t), &\mbox{for } (x,t)\in Q_\de,\\
			0, &\mbox{for }(x,t)\in Q-Q_\de, 
		\end{cases}
	\end{equation*}
	then $Ey_\de\in L^2(Q)$ and $\|Ey_\de\|_{L^2(Q)}=\|y_\de\|_{L^2(Q_\de)}$. 
	
	If $y_\de\in L^2(0,T;H_0^1(\Om_\de;w_\de))$, then $Ey_\de\in L^2(0,T; H_0^1(\Om;w))$ and from 
	\begin{equation*}
		\nabla Ey_\de=
		\begin{cases}
			\nabla y_\de, &\mbox{on }Q_\de,\\
			0, &\mbox{on }Q-Q_\de, 
		\end{cases}
	\end{equation*}
	we get  
	\begin{equation*}
		\begin{split}
			\|Ey_\de\|_{L^2(0,T; H_0^1(\Om;w))}^2
			&=\iint_{Q}\nabla Ey_\de\cdot A\nabla Ey_\de\df x\df t\\
			&=\iint_{Q_\de}\nabla y_\de \cdot A_\de \nabla y_\de \df x\df t=\|y_\de\|_{L^2(0,T; H_0^1(\Om_\de;w_\de))}^2. 
		\end{split}
	\end{equation*}
	
	If $y_\de\in H^1(0,T; L^2(\Om_\de))$, then $Ey_\de\in H^1(0,T; L^2(\Om))$, and 
	\begin{equation*}
		\pt_tEy_\de=E\pt_ty_\de, \quad \|\pt_t Ey_\de\|_{L^2(Q)}=\|E\pt_ty_\de\|_{L^2(Q)}. 
	\end{equation*}
	Indeed, for each $\zeta\in C_0^\iy(0,T), v\in C_0^\iy(\Om)$, we have 
	\begin{equation*}
		\begin{split}
			\iint_Q(\pt_tEy_\de)v\zeta\df x \df t
			&=-\iint_Q (Ey_\de)v\pt_t\zeta\df x\df t=-\iint_{Q_\de} y_\de v \pt_t\zeta\df x\df t\\
			&=\iint_{Q_\de} (\pt_ty_\de)v\zeta\df x\df t=\iint_Q (E\pt_ty_\de)v\zeta \df x\df t, 
		\end{split}
	\end{equation*}
	then $\pt_tEy_\de=E\pt_ty_\de$ in the sense of distribution, and hence $\pt_tEy_\de=E\pt_ty_\de$ in $L^2(Q)$.
	
	If $y_\de\in H^2(0,T; L^2(\Om_\de))$, then $Ey_\de\in H^2(0,T; L^2(\Om))$, and $\|\pt_{tt}Ey_\de\|_{L^2(Q)}=\|E\pt_{tt}y_\de\|_{L^2(Q)}$.  
\end{notation}

\begin{theorem}\label{03.13.T2}
	Let $y^0\in C_0^\iy(\Om)$ and $y^1\in L^2(\Om)$ and $f\in L^2(Q)$.  Let $y$ be the weak  solution of \eqref{01.10.1} with respect to $(y^0,y^1,f)$, $y_\de$ be the weak solution of \eqref{03.13.4} with respect to $(y^0|_{\Om_\de}, y^1|_{\Om_\de},f|_{Q_\de})$ for $\de\in (0,\f{1}{4})$. Then, by abstract subsequence, we have
	\begin{equation}\label{03.13.15}
		\begin{split}
			Ey_\de
			&\ra y \hspace{3mm}\mbox{ weakly in } L^2(0,T; H_0^1(\Om;w)), \\
			\pt_t Ey_\de 
			&\ra \pt_t y \mbox{ weakly in } L^2(Q). 
		\end{split}
	\end{equation}
	And hence 
	\begin{equation}\label{03.13.16}
		Ey_\de\ra y\mbox{ strongly in } L^2(Q). 
	\end{equation}
	Moreover, if in addition $y^1\in C_0^\iy(\Om)$ and $f\in H^1(0,T; L^2(\Om))$, then for each fixed $\de\in (0,\f{1}{4})$ we have 
	\begin{equation}\label{03.13.17}
		\f{\pt Ey_\ga}{\pt \nu}\ra \f{\pt y}{\pt\nu} \mbox{ strongly in } L^2(0,T ; L^2(\Ga_\de)) \mbox{ as } \ga\ra 0^+, 
	\end{equation}
	and 
	\begin{equation}\label{03.14.9}
		\left\|\f{\pt y}{\pt \nu}\right\|_{L^2(0,T; L^2(\Ga))}\leq C\left(\|y^0\|_{H_0^1(\Om;w)}+\|y^1\|_{L^2(\Om)}+\|f\|_{L^2(Q)}\right), 
	\end{equation}
	where the positive constant $C$ depending only on $\al$ and $T$. 
\end{theorem}

\begin{proof}
	We prove this theorem by the following steps. 
	
	{\it Step 1}. Weakly convergence. 
	
	By abstract subsequence, we only need to show $\de\in (0,\de_0)$ with $\de_0=\min\{\f{1}{4}, \dist(\supp y^0, \pt\Om)\}$. Then $y^0\in C_0^\iy(\Om_\de)$. Hence, from \eqref{03.13.7} in  Theorem \ref{03.13.T1} and  Notation \ref{03.13.N2}, we get
	\begin{equation*}
		\begin{split}
			\esssup_{t\in[0,T]}\left(\|Ey_\de(t)\|_{H_0^1(\Om;w)}+\|\pt_tEy_\de\|_{L^2(\Om)}\right) 
			&\leq C\left(\|y^0\|_{H_0^1(\Om;w)}+\|y^1|_{\Om_\de}\|_{L^2(\Om_\de)}+\|f|_{Q_\de}\|_{L^2(Q_\de)}\right)\\
			&\leq C\left(\|y^0\|_{H_0^1(\Om;w)}+\|y^1\|_{L^2(\Om)}+\|f\|_{L^2(Q)}\right), 
		\end{split}
	\end{equation*}
	where the positive constants $C$ depending only on $\al$ and $T$. Then there exists a subsequence of $\{Ey_\de\}_{\de\in(0,\de_0)}$ and 
	\begin{equation}\label{03.14.3}
		z\in L^2(0,T; H_0^1(\Om;w))\cap H^1(0,T; L^2(\Om))
	\end{equation}
	such that 
	\begin{equation}\label{03.14.2}
		\begin{split}
			Ey_\de
			&\ra z \hspace{3mm}\mbox{ weakly in } L^2(0,T; H_0^1(\Om;w)), \\
			\pt_t Ey_\de 
			&\ra \pt_t z \mbox{ weakly in } L^2(Q). 
		\end{split}
	\end{equation}
	
	{\it Step 2}. We prove $z=y$, i.e., $z$ is a weak solution of \eqref{01.10.1} with respect to $(y^0, y^1, f)$. 
	
	Let $\psi\in C^\iy(\ol Q)$ with $\supp\psi(t)\s \Om$ for all $t\in [0,T]$ and $\psi(T)=\pt_t\psi(T)=0$. Set 
	\begin{equation*}
		\wh \de_0=\f{1}{2}\min\left\{\de_0, \min_{t\in [0,T]}\dist(\psi(t), \pt\Om)\}\right\}. 
	\end{equation*}
	Then, for each $\de\in (0,\wh\de_0)$, we have $\psi\in C^\iy(\ol Q_\de)$ with $\supp\psi(t)\s \Om_\de$ for all $t\in [0,T]$ and $\psi(T)=\pt_t\psi(T)=0$, and hence 
	\begin{equation*}
		\begin{split}
			&\int_0^T (y_\de, \pt_{tt}\psi)_{L^2(\Om_\de)}\df t+\int_0^T(y_\de, \psi)_{H_0^1(\Om_\de;w_\de)}\df t\\
			&=\int_0^T(f_\de,\psi)_{L^2(\Om_\de)}\df t-(y_\de^0, \pt_t\psi(0))_{L^2(\Om_\de)}+(y_\de^1,\psi(0))_{L^2(\Om_\de)}
		\end{split}
	\end{equation*}
	by \eqref{03.14.1}. i.e., 
	\begin{equation*}
		\begin{split}
			&\int_0^T (Ey_\de, \pt_{tt}\psi)_{L^2(\Om)}\df t+\int_0^T(Ey_\de, \psi)_{H_0^1(\Om;w)}\df t\\
			&=\int_0^T(f,\psi)_{L^2(\Om)}\df t-(y^0, \pt_t\psi(0))_{L^2(\Om)}+(y^1,\psi(0))_{L^2(\Om)}.
		\end{split}
	\end{equation*}
	Together with this and \eqref{03.14.2}, we get
	\begin{equation*}
		\begin{split}
			&\int_0^T (z, \pt_{tt}\psi)_{L^2(\Om)}\df t+\int_0^T(z, \psi)_{H_0^1(\Om;w)}\df t\\
			&=\int_0^T(f,\psi)_{L^2(\Om)}\df t-(y^0, \pt_t\psi(0))_{L^2(\Om)}+(y^1,\psi(0))_{L^2(\Om)}. 
		\end{split}
	\end{equation*}
	Combining this and Proposition \ref{03.13.P1} and \eqref{03.14.3}, we  show that $z$ is a weak solution of \eqref{01.10.1} with respect to $(y^0,y^1,f)$. 
	
	{\it Step 3}. We prove \eqref{03.13.16}. 
	
	Note that the embeddings  $H_0^1(\Om;w)\hra L^2(\Om)$ and 
	\begin{equation*}
		\left\{\xi\in L^2(0,T; H_0^1(\Om;w))\colon \pt_t \xi \in L^2(Q)\right\}\hra L^2(Q)
	\end{equation*}
	are compact, then, by abstract subsequence, we obtain 
	\begin{equation*}
		Ey_\de\ra y \mbox{ strongly in } L^2(Q). 
	\end{equation*}
	
	{\it Step 4}. We prove \eqref{03.13.17}. 
	Let $y^0,y^1\in C_0^\iy(\Om)$ and $f\in H^1(0,T; L^2(\Om))$. 
	
	Set $\wt \de_0=\f{1}{2}\min\{\f{1}{4}, \dist(\supp y^0,\pt\Om), \dist(\supp y^1,\pt\Om)\}$. Let $\de\in (0,\wt\de_0)$. For all $\ga\in (0,\de)$, from \eqref{03.13.8}  and $y^0\in D(\mcA_\ga), y^1\in H_0^1(\Om_\ga;w_\ga)$ and $f_\ga\in H^1(0,T; L^2(\Om_\ga))$, we obtain 
	\begin{equation}\label{03.14.4}
		\begin{split}
			\left\|y_\ga\right\|_{L^2(0,T; H^2(\Om_\de))}+\|\pt_{tt}y_\ga\|_{L^2(Q_\de)}
			&\leq C_\de\left(\|y^0\|_{H_0^1(\Om;w)}+\|y^1\|_{L^2(\Om)}+\|f\|_{L^2(Q)}\right)
		\end{split}
	\end{equation}
	and 
	\begin{equation}\label{03.14.5}
		\begin{split} 
		&\left\|\pt_t\f{\pt y_\ga}{\pt\nu}\right\|_{L^2(0,T; L^2(\Ga_\ga))}=\left\|\f{\pt (\pt_ty_\ga)}{\pt\nu}\right\|_{L^2(0,T; L^2(\Ga_\ga))}\\
		&\leq C\left(\|y^0\|_{D(\mcA)}+\|y^1\|_{H_0^1(\Om;w)}+\|f\|_{H^1(0,T; L^2(\Om))}\right), 
		\end{split} 
	\end{equation}
	where $C>0$ depends only on $\alpha$ and $T$, and $C_\de>0$ depends only on $\alpha, T$, and $\de$ (see Remark \ref{03.14.R1}). It may happen that $C_\de \to \infty$ as $\de \to 0^+$. By the classical Sobolev embedding theorem and \eqref{03.14.4}, we get
	\begin{equation}\label{03.14.8}
		\left\|\f{\pt y_\ga}{\pt\nu}\right\|_{L^2(0,T; H^\f{1}{2}(\Ga_\de))}\leq C_\de \left(\|y^0\|_{H_0^1(\Om;w)}+\|y^1\|_{L^2(\Om)}+\|f\|_{L^2(Q)}\right). 
	\end{equation}
	Note that the embeddings $H^\f{1}{2}(\Ga_\de)\hra L^2(\Ga_\de)$ and 
	\begin{equation*}
		\left\{\xi\in L^2(0,T; H^\f{1}{2}(\Ga_\de))\colon \pt_t\xi \in L^2(0,T; L^2(\Ga_\de))\right\}\hra L^2(0,T; L^2(\Ga_\de))
	\end{equation*}
	are compact, then, by abstract subsequence, from \eqref{03.14.5} and  \eqref{03.14.8}, there exists $h_\de\in L^2(0,T; L^2(\Ga_\de))$ such that 
	\begin{equation}\label{03.14.6}
		\f{\pt y_\ga}{\pt \nu_A}\bigg|_{\Ga_\de\ts(0,T)}=\f{\pt y_\ga}{\pt \nu}\bigg|_{\Ga_\de\ts(0,T)}\ra h_\de \mbox{ in }L^2(0,T; L^2(\Ga_\de)). 
	\end{equation}
	
	Let $\Psi\in C_0^\iy(\Ga_\de)$. Then there exists   $\wt\Psi\in C^\iy(\overline \Om_\de)$ such that $\wt\Psi =\Psi$ on $\Ga_\de$, $\wt\Psi =0 $ on  $\Ga_2 ^\de$, and $\wt\Psi =0$ on $\Om \setminus \Om_\de$. For each $\ga\in (0,\de)$, multiplying $\wt\Psi$ on the both sides of \eqref{03.13.4}, integrating on $Q_\ga$, integration by parts, we get 
	\begin{equation*}
		\begin{split}
			\iint_{Q_\ga}f_\ga \wt \Psi\df x\df t=\iint_{Q_\ga}(\pt_{tt}y_\ga)\wt\Psi\df x\df t-\iint_{\Ga_\de\ts (0,T)} \left(A_\ga \nabla y_\ga\cdot \nu\right) \Psi\df S\df t+\iint_{Q_\ga}A_\ga \nabla y_\ga \cdot \nabla \wt \Psi\df x\df t. 
		\end{split}
	\end{equation*}
	i.e., 
	\begin{equation*}
		\begin{split}
			\iint_Q f\wt \Psi\df x\df t=\iint_Q (\pt_{tt}Ey_\ga)\wt\Psi\df x\df t-\iint_{\Ga_\de\ts (0,T)}(A\nabla Ey_\ga\cdot \nu)\Psi\df S\df t+\iint_Q A\nabla Ey_\ga \cdot \nabla \wt \Psi\df x\df t. 
		\end{split}
	\end{equation*}
	From \eqref{03.14.2} and  \eqref{03.14.4} and  $z=y$, we get 
	\begin{equation*}
		\pt_{tt}Ey_\de\ra \pt_{tt}y \mbox{ weakly in }L^2(Q), 
	\end{equation*}
	and then 
	\begin{equation}\label{03.14.7}
		\iint_Q f\wt\Psi\df x\df t=\iint_Q (\pt_{tt}y)\wt\Psi\df x\df t-\iint_{\Ga_\de\ts (0,T)}h_\de\Psi\df S\df t+\iint_Q A\nabla y\cdot \nabla \wt\Psi\df x\df t
	\end{equation}
	by \eqref{03.14.2} and letting $\ga\ra 0^+$. Multiplying $\wt\Psi$ on the both sides of \eqref{01.10.1}, integrating on $Q$, integration by parts (note that $\wt\Psi=0$ on $\Om-\Om_\de$), we get
	\begin{equation*}
		\iint_Q f\wt\Psi\df x\df t=\iint_Q (\pt_{tt}y)\wt \Psi\df x\df t-\iint_{\pt Q}(A\nabla y\cdot \nu)\Psi\df S\df t+\iint_Q A\nabla y\cdot \nabla \wt\Psi\df x\df t. 
	\end{equation*}
	Together this with \eqref{03.14.7}, we get
	\begin{equation*}
		\iint_{\Ga_\de\ts (0,T)}h_\de\Psi\df S\df t=\iint_{\pt Q}(A\nabla y\cdot \nu)\Psi\df S\df t. 
	\end{equation*}
	This shows that $h_\de=\f{\pt y}{\pt\nu_A}=\f{\pt y}{\pt\nu}$ on $\Ga_\de\ts (0,T)$ by the arbitrary of $\Psi$. This proves \eqref{03.13.17}
	
	From the fourth term in \eqref{03.13.7}, we get
	\begin{equation*}
		\left\|\f{\pt y_\de}{\pt\nu}\right\|_{L^2(0,T; L^2(\Ga_\ga))}\leq C\left(\|y^0\|_{H_0^1(\Om;w)}+\|y^1\|_{L^2(\Om)}+\|f\|_{L^2(Q)}\right), 
	\end{equation*}
	where the positive constant $C$ depending only on $\al$ and $T$. 
	Combining this and  \eqref{03.14.6} and $h_\de=\f{\pt y}{\pt \nu}$ on $\Ga_\de\ts (0,T)$, we get
	\begin{equation*}
		\left\|\f{\pt y}{\pt \nu}\right\|_{L^2(0,T; L^2(\Ga_\de))}\leq C\left(\|y^0\|_{H_0^1(\Om;w)}+\|y^1\|_{L^2(\Om)}+\|f\|_{L^2(Q)}\right), 
	\end{equation*}
	where the positive constant $C$ depending only on $\al$ and $T$. Letting $\de\ra 0^+$ we obtain \eqref{03.14.9}. 
\end{proof}

\begin{remark}\label{03.15.R1}
	By Theorem \ref{03.13.T2}, the hidden regularity of equation \eqref{01.10.1} can be reduced to that of equation \eqref{03.13.4}. However, this reduction cannot be obtained directly from Theorem \ref{03.11.T1}.
	
	From the proof of Theorem \ref{03.13.T2}, we note that this may  not hold: 
	\begin{equation*}
		\f{\pt Ey_\ga}{\pt \nu}\ra \f{\pt y}{\pt \nu} \mbox{ strongly in } L^2(0,T; L^2(\Ga)) \mbox{ as } \ga \ra 0^+. 
	\end{equation*}
\end{remark}

\begin{proposition}\label{03.16.P1}
	Let $y^0\in H_0^1(\Om;w), y^1\in L^2(\Om)$ and $f\in L^2(Q)$. Let $y$ be the weak solution of \eqref{01.10.1} with respect to $(y^0,y^1,f)$.  Then 
	\begin{equation*}
		\left\|\f{\pt y}{\pt \nu}\right\|_{L^2(0,T; L^2(\Ga))}\leq C\left(\|y^0\|_{H_0^1(\Om;w)}+\|y^1\|_{L^2(\Om)}+\|f\|_{L^2(Q)}\right), 
	\end{equation*}
	where the positive constant $C$ depending only on $\al$ and $T$. 
\end{proposition}

\begin{proof}
	Choosing $y_n^0\in C_0^\iy(\Om), y_n^1\in C_0^\iy(\Om)$ and $f_n\in H^1(0,T; L^2(\Om))$ such that 
	\begin{equation*}
		\begin{split}
			y_n^0\ra y^0 \mbox{ strongly in }H_0^1(\Om;w), \quad y_n^1\ra y^1\mbox{ strongly in } L^2(\Om),\quad f_n\ra f\mbox{ strongly in } L^2(Q). 
		\end{split}
	\end{equation*}
	Let $y_n
	\ (n\in\N)$ be the weak solution of \eqref{01.10.1} with respect to $(y_n^0,y_n^1,f_n)$, and $y$ be the weak solution of \eqref{01.10.1} with respect to $(y^0,y^1,f)$. 
	From \eqref{03.14.9}, we obtain 
	\begin{equation*}
		\left\|\f{\pt y_n}{\pt\nu}\right\|_{L^2(0,T; L^2(\Ga))}\leq C\left(\|y^0\|_{H_0^1(\Om;w)}+\|y^1\|_{L^2(\Om)}+\|f\|_{L^2(Q)}\right), 
	\end{equation*}
	where the positive constant $C$ depending only on $\al$ and $T$. And from \eqref{03.12.12} and \eqref{03.16.1}, we get
	\begin{equation*}
		\begin{split}
			\left\|\f{\pt (y-y_n)}{\pt\nu}\right\|_{L^2(0,T; L^2(\Ga_\de))}\leq C\left(\|y^0-y_n^0\|_{H_0^1(\Om;w)}^2+\|y^1-y_n^1\|_{L^2(\Om)}+\|f_n-f\|_{L^2(Q)}\right), 
		\end{split}
	\end{equation*}
	where the positive constants $C$ depending only on $\al, T$ and $\de$. Hence 
	\begin{equation*}
		\left\|\f{\pt y}{\pt\nu}\right\|_{L^2(0,T; L^2(\Ga_\de))}\leq C\left(\|y^0\|_{H_0^1(\Om;w)}+\|y^1\|_{L^2(\Om)}+\|f\|_{L^2(Q)}\right)
	\end{equation*}
	for all $\de\in (0,\f{1}{8})$, 
	where $C>0$ depends only on $\alpha$ and $T$ and is independent of $n$. Letting $\de\ra 0^+$, we get the proposition. 
\end{proof}

\end{document}